\documentclass{./Article-eng}

\usepackage[all]{xy}
\usepackage{hyperref}
\hypersetup{colorlinks,allcolors=blue}
\usepackage{tikz-cd}%implies
\makeatletter

\newcommand{\Rmnum}[1]{\expandafter\@slowromancap\romannumeral #1@}
\makeatother

\let\Gamma\varGamma
\let\Delta\varDelta

\begin{document}
%%%% Titre
\title[Alg. fib. spaces w/ strictly nef rel. anti-log can. divisor]{Algebraic fibre spaces with strictly nef relative anti-log canonical divisor}

%\title{On singular projective varieties with strictly nef anti-log canonical divisor: the case of log terminal singularities}
%\author{
%Jie Liu\thanks{Institute of Mathematics, Academy of Mathematics and Systems Science, Chinese Academy of Sciences, Beijing, 100190, China. E-mail: {\tt jliu@amss.ac.cn}},\; Wenhao Ou\thanks{Institute of Mathematics, Academy of Mathematics and Systems Science, Chinese Academy of Sciences, Beijing, 100190, China. E-mail: {\tt wenhaoou@amss.ac.cn}},\; Juanyong Wang\thanks{HCMS, Academy of Mathematics and Systems Science, Chinese Academy of Sciences, Beijing, 100190, China. E-mail: {\tt juanyong.wang@amss.ac.cn}},\; Xiaokui Yang\thanks{Department of Mathematics and Yau Mathematical Sciences Center, Tsinghua University, Beijing, 100084, China. E-mail: {\tt xkyang@mail.tsinghua.edu.cn}},\; Guolei Zhong\thanks{Department of Mathematics, National University of Singapore, 10 Lower Kent Ridge Road, Singapore 119076, Republic of Singapore. E-mail: {\tt zhongguolei@u.nus.edu}}
%}

\author{
Jie Liu
%\thanks{jliu@amss.ac.cn}
}
\address{
Jie Liu\\
Institute of Mathematics, Academy of Mathematics and Systems Science, Chinese Academy of Sciences, Beijing, 100190, China}
\email{jliu@amss.ac.cn}

\author{
Wenhao Ou
%\thanks{wenhaoou@amss.ac.cn}
}
\address{
Wenhao Ou\\
Institute of Mathematics, Academy of Mathematics and Systems Science, Chinese Academy of Sciences, Beijing, 100190, China}
\email{wenhaoou@amss.ac.cn}

\author{
Juanyong Wang
%\thanks{juanyong.wang@amss.ac.cn}
}
\address{
Juanyong Wang\\
HCMS, Academy of Mathematics and Systems Science, Chinese Academy of Sciences, Beijing, 100190, China.}
\email{juanyong.wang@amss.ac.cn}
%\affil{Institute of Mathematics, Academy of Mathematics and Systems Science, Chinese Academy of Sciences, Beijing, 100190, China}

\author{Xiaokui Yang
%\thanks{xkyang@mail.tsinghua.edu.cn}
}
\address{
Xiaokui Yang\\
Department of Mathematics and Yau Mathematical Sciences Center, Tsinghua University, Beijing, 100084, China}
\email{xkyang@mail.tsinghua.edu.cn}
%\affil{Department of Mathematics and Yau Mathematical Sciences Center, Tsinghua University, Beijing, 100084, China}

\author{Guolei Zhong
%\thanks{zhongguolei@u.nus.edu}
}
\address{
Guolei Zhong\\
Department of Mathematics, National University of Singapore, 10 Lower Kent Ridge Road, Singapore 119076, Republic of Singapore}
\email{zhongguolei@u.nus.edu}
%\affil{Department of Mathematics, National University of Singapore, 10 Lower Kent Ridge Road, Singapore 119076, Republic of Singapore}

\date{}

%\begin{document}
\maketitle
%\paragraph{Abstract:} 
\begin{abstract}
Let $(X,\Delta)$ be a projective klt pair, and  $f:X\to Y$  a fibration to a smooth projective variety $Y$ with strictly nef relative anti-log canonical divisor $-(K_{X/Y}+\Delta)$.
We prove that $f$ is a locally constant fibration  with rationally connected fibres, and the base  $Y$ is a canonically polarized  hyperbolic projective manifold. In particular,  when  $Y$ is a single point, we establish that  $X$ is rationally connected. 
Moreover, when $\dim X=3$ and $-(K_X+\Delta)$ is strictly nef, we prove that $-(K_X+\Delta)$ is ample, which confirms the singular version of the Campana-Peternell conjecture for threefolds.

\end{abstract}
%\newline
\tableofcontents

\section{Introduction}
\label{sec_intro}
%\addcontentsline{toc}{section}{Introduction}

In this paper, we work over the field  $\mathbb{C}$ of complex numbers. 
Let $X$ be a projective variety.  
Recall that a $\QQ$-Cartier divisor $L$ over $X$ is said to be \emph{strictly nef} (resp. \emph{nef}) if  $L\cdot C>0$ (resp. $L\cdot C\geqslant 0$) for every (complete) curve $C$ on $X$. 
It is clear that ample divisors are always strictly nef, but strictly nef divisors are not necessarily ample (e.g. \cite[Example 10.6]{Har70}). It is still a challenge to
investigate the relationship between these two concepts.  As a major problem along this line,  Campana and Peternell conjectured
in \cite{CP91} that the  strict nefness and ampleness are equivalent for the anti-canonical divisor $-K_X$ of a smooth projective variety $X$.
In the 1990s, this conjecture was intensively studied in lower dimensions and has been confirmed in dimension $\leqslant 3$ (cf.~\cite{Mae93} and \cite{Ser95}).
The singular version of this conjecture has also been proved for threefolds with at worst canonical singularities in \cite[Theorem 3.9]{Ueh00}. For more results motivated by the Campana-Peternell conjecture, we  refer the readers to \cite{CCP08, LOY19, LOY20, LOY21}   and the references therein.\\

In this paper, we  study the geometric structures of a projective klt pair $(X,\Delta)$ admitting a fibration $X\to Y$ with strictly nef  relative anti-log canonical divisor $-(K_{X/Y}+\Delta)$. 
First, we propose the following question, which is the singular version of Campana and Peternell's conjecture:

%, i.e., $X$ is a normal projective variety, $\Delta$ is an effective $\mathbb{Q}$-divisor such that $K_X+\Delta$ %is $\mathbb{Q}$-Cartier, and $(X,\Delta)$ has only klt singularities. 

\begin{conj}\label{main-conj-singular-ample}
Let $(X,\Delta)$ be a projective klt pair.
If $-(K_X+\Delta)$ is strictly nef, 
then $-(K_X+\Delta)$ is  ample.
\end{conj}

\noindent 
Recently, Li, the second author and the fourth author proved in \cite{LOY19} that  smooth projective varieties with strictly nef   anti-canonical divisors  are rationally connected (\cite[Theorems 1.2 and 1.3]{LOY19}). By the well-known result of Campana and  Koll\'ar-Mori-Miyaoka (\cite{Cam92,KoMM92}),  this provides important evidences for an affirmative answer to the Campana and Peternell conjecture, i.e.,  the smooth case of {\hyperref[main-conj-singular-ample]{Conjecture \ref*{main-conj-singular-ample}}} in all dimensions.
As motivated by  \cite[Theorem 1.3]{LOY19}, we consider a particular case of {\hyperref[main-conj-singular-ample]{Conjecture \ref*{main-conj-singular-ample}}} in the singular pair setting (cf.~\cite{Zhang06}). 

\begin{conj}\label{main-conj-singular}
Let $(X,\Delta)$ be a projective klt pair.
If $-(K_X+\Delta)$ is strictly nef, 
then $X$ is  rationally connected.
\end{conj}

\noindent Before  giving  an affirmative answer of {\hyperref[main-conj-singular]{Conjecture \ref*{main-conj-singular}}}, we recall that the augmented irregularity $q^\circ(X)$ of a normal projective variety $X$ is defined as the maximum of the irregularities $q(\widetilde{X}):=\dimcoh^1(\widetilde{X},\scrO_{\widetilde{X}})$, where $\widetilde{X}\to X$ runs over the (finite) quasi-\'etale covers of $X$ (cf.~\cite[Definition 2.6]{NZ10}). The first main result of this paper is:
%Further, we shall give a positive answer to  {\hyperref[main-conj-singular-ample]{Question \ref*{main-conj-singular-ample}}} in dimension three.

\begin{mainthm}
\label{mainthm-RC}
Let $X$ be a normal projective variety. Suppose that there is an effective $\QQ$-divisor $\Delta$ on $X$ such that the pair $(X,\Delta)$ is klt and that $-(K_X+\Delta)$ is strictly nef. Then $X$ is rationally connected.
In particular, the augmented irregularity $q^\circ(X)=0$.
\end{mainthm}

In the following, we  study the relative case of the above setting by considering a fibration $f:X\rightarrow Y$  between projective manifolds.  
 Miyaoka proved in \cite[Theorem 2]{Miy93} that, the relative anti-canonical divisor $-K_{X/Y}$ is never ample if $f$ is  smooth and $Y$ is not a single point. 
In the past few decades, many generalizations have been established  by dropping the smoothness of $X$ and $f$ (e.g.  \cite{Zhang96}, \cite{ADK08}, \cite{AD13}). A remarkable one 
 is the following result in \cite[Theorem 5.1]{AD13} established by  Araujo and Druel :

\begin{thm}[Araujo-Druel]\label{t.AD}
Let $X$ be a normal projective variety and let $f:X\rightarrow C$ be a fibration onto a smooth projective curve $C$.
Assume that there exists an effective $\mathbb{Q}$-Weil divisor $\Delta$ on $X$ such that $K_{X}+\Delta$ is $\mathbb{Q}$-Cartier.
	\begin{enumerate}
		\item[\rm(1)] If $(X,\Delta)$ is log canonical over the generic point of $C$, then $-(K_{X/C}+\Delta)$ is not ample.
		
		\item[\rm(2)] If $(X,\Delta)$ is klt over the generic point of $C$, then $-(K_{X/C}+\Delta)$ is not nef and big.
	\end{enumerate}
\end{thm}

\noindent One may wonder whether $-K_{X/Y}$ can be strictly nef in the relative setting.  Indeed, there is  an example constructed by  Mumford.  Let $C$ be a smooth  curve of genus at least two. 
There exists a vector bundle $E$ of rank two over $C$ such that the first Chern class $c_1(E)=0$ and the tautological bundle $\scrO_{\PP E}(1)$ of the projectivization $X\coloneqq \PP E$ is strictly nef. It is easy to see that $-K_{X/C}$ is strictly nef.
The second main result in this paper is to investigate the geometric structures for such fibrations.
More precisely, we obtain:

\begin{mainthm}\label{t.mainthm}
	Let $f:X\rightarrow Y$ be a fibration from a normal projective variety $X$ onto a smooth projective variety $Y$. Suppose that there exists an effective $\QQ$-Weil divisor $\Delta$ on $X$ such that $(X,\Delta)$ is a klt pair and $-(K_{X/Y}+\Delta)$ is strictly nef. 
	Then $f$ is a locally constant fibration with rationally connected fibres and $Y$ is a canonically polarized hyperbolic projective manifold.
\end{mainthm}

Recall that, a projective manifold $Y$ is said to be \textit{canonically polarized} if its canonical divisor $K_Y$ is ample and to be \textit{hyperbolic} if any holomorphic map $\mathbb{C}\to Y$ from the complex line $\mathbb{C}$ is constant. 
As an application of {\hyperref[t.mainthm]{Theorem \ref*{t.mainthm}}} and techniques developed in \cite{Ou21} (see also \cite[Theorem 1.1]{Druel17}), we obtain characterizations of the geometric structures of regular foliations on projective manifolds with strictly nef anti-canonical divisor. 

\begin{maincor}
	\label{c.algebraicallyintegrable}
	Let $\scrF$ be a foliation on a projective manifold $X$. Assume that either $\scrF$ is regular, or $\scrF$ has a compact leaf. 
	
	\begin{enumerate}
		\item[\rm(1)] If $-K_{\scrF}$ is strictly nef, then there exists a locally constant fibration $f:X\rightarrow Y$ with rationally connected fibres over a canonically polarized hyperbolic projective manifold $Y$ such that $\scrF$ is induced by $f$. In particular, $\scrF$ is algebraically integrable.
		
		\item[\rm(2)] If the fundamental group $\pi_1(X)$ is virtually solvable (i.e., a finite index subgroup of $\pi_1(X)$ is solvable), then $-K_{\scrF}$ cannot be strictly nef.
	\end{enumerate}	
\end{maincor}

%{\hyperref[mainthm-RC]{Theorem \ref*{mainthm-RC}}} $\sim$ {\hyperref[mainthm-ample-3klt]{Theorem \ref*{mainthm-ample-3klt}}} below are our main results of this paper.

Finally, we show that {\hyperref[main-conj-singular-ample]{Conjecture \ref*{main-conj-singular-ample}}} holds when $\dim X=3$, which extends results in  \cite[Theorem 3.9]{Ser95} and \cite[Main Theorem]{Ueh00} (see also \cite[Conjecure 1.2]{HL20}).

\begin{mainthm}
\label{mainthm-ample-3klt}
Let $X$ be a normal projective threefold. If there is an effective $\QQ$-divisor $\Delta$ on $X$ such that the pair $(X,\Delta)$ is klt and that $-(K_X+\Delta)$ is strictly nef, then $-(K_X+\Delta)$ is ample.
\end{mainthm}

\noindent 
To  end up this section, we propose a singular version of Serrano's conjecture (\cite{Ser95}), which has an affirmative answer when $\dim X\leqslant 2$ (cf.~\cite[Corollary 1.8]{HL20}). 
\begin{ques}\label{main-conj-singular-arbitrary}
Let $(X,\Delta)$ be a projective klt pair, and $L$ be a strictly nef $\mathbb{Q}$-divisor on $X$.
Is $K_X+\Delta+tL$  ample for sufficiently large $t\gg 1$?
\end{ques}

%Indeed,
%\begin{prop}[{\cite[Corollary 1.8]{HL20}; see {\hyperref[prop_Q-Goren_surface]{Proposition \ref*{prop_Q-Goren_surface}}}   for a further extension}]\label{main_thm_surface}
%Let $(X,\Delta)$ be a projective klt surface.
%Suppose $L_X$ is a strictly nef Cartier divisor on $X$.
%Then $K_X+\Delta+tL_X$ is ample for every real number $t>3$.
%\end{prop}

\paragraph{\textbf{Acknowledgement}}
J. Liu is supported by the NSFC grants No. 11688101 and No. 12001521. 
J. Wang is supported by the National Natural Science Foundation of China project `Geometry, analysis, and computation on manifolds'.
X. Yang is supported by the NSFC grant 12171262.
G. Zhong is supported by a President's Scholarship of NUS.

\vskip 2\baselineskip

\section{Preliminary results}
\label{sec_pre}
Throughout this paper, we refer to \cite[Chapter 2]{KM98} for different kinds of singularities.
By a \textit{projective klt (resp. canonical, dlt) pair} $(X,\Delta)$, we mean that $X$ is a normal projective variety, $\Delta\geqslant 0$ is an effective $\mathbb{Q}$-divisor such that $K_X+\Delta$ is $\mathbb{Q}$-Cartier, and  $(X,\Delta)$ has only klt (resp. canonical, dlt) singularities.

A $\mathbb{Q}$-Cartier divisor $L$ on a projective variety is said to be \textit{strictly nef} if $L\cdot C>0$ for every (complete) curve $C$ on $X$.
We also recall the definition of almost strictly nef divisors for the convenience of the proofs in the later sections.
\begin{defn}[{cf.~\cite[Definition 1.1]{CCP08}}]\label{defn-almost-sn}
A $\mathbb{Q}$-Cartier divisor $L$ on a normal projective variety $X$ is called  \textit{almost strictly nef}, if there exist a birational morphism $\pi:X\to Y$ to a  normal projective variety $Y$ and a strictly nef $\mathbb{Q}$-divisor $L_Y$ on $Y$ such that $L=\pi^*L_Y$.  
\end{defn}
One of the motivation to introduce almost strictly nef divisors is to study the  non-uniruled case of {\hyperref[main-conj-singular-arbitrary]{Question \ref*{main-conj-singular-arbitrary}}}  via the descending of Iitaka fibrations, in which case, the pullback of a strictly nef divisor to a higher model which resolves the indeterminacy is almost strictly nef  (cf.~\cite[Theorem 2.6]{CCP08}).

Now, we show  that, under the condition of  {\hyperref[main-conj-singular]{Question \ref*{main-conj-singular}}},  such $X$ is uniruled.
This permits us to take a first glance on the geometric property of a  projective klt pair with (almost) strictly nef anti-log canonical divisor. 
%We also note that, {\hyperref[main-conj-singular]{Question \ref*{main-conj-singular}}} has an affirmative answer when $X$ is smooth with the boundary $\Delta=0$ (cf.~ \cite[Theorem 1.3]{LOY19}).

\begin{prop}\label{prop_strict_uniruled}
Let $(X,\Delta)$ be a projective pair such that $-(K_X+\Delta)$ is almost strictly nef.
Then $X$ is uniruled, i.e., $X$ is covered by rational curves.	
\end{prop}
\begin{proof}
Suppose the contrary that $X$ is not uniruled.
Then $K_X$ is pseudo-effective as a Weil divisor (cf. \cite{BDPP13}) by considering a resolution of $X$.
In other words, $K_X$ lies in the closure of the cone generated by effective Weil divisors on $X$ (cf.~\cite[Definition 2.2]{MZ18}).   
%\cite[Chapter II, Definition 5.5]{Nak04} and
So $-\Delta=-(K_X+\Delta)+K_X$ is also pseudo-effective (as a 
Weil divisor) and thus $\Delta=0$.
This in turn implies the almost strict nefness of $-K_X$, a contradiction to the pseudo-effectivity of $K_X$. 
\end{proof}

In the remaining part of this section, we review the definition and basic properties of locally constant fibrations and flat vector bundles for the convenience of readers.

Let $\phi:X\rightarrow Y$ be a proper morphism between normal varieties. We say that $\phi$ is a \emph{fibration} if $\phi_*\scrO_X=\scrO_Y$. 
\begin{defn}[{\cite[Definition 2.6]{MW21}; cf.~\cite[Definition 1.6]{Wang20}}]\label{defn-locally-constant}
Let $\phi:X\to Y$ be a surjective morphism with connected fibres between analytic varieties, and let $\Delta$ be a Weil $\mathbb{Q}$-divisor on $X$.
\begin{enumerate}
	\item[(1)] The morphism $\phi:X\to Y$ is said to be a \textit{locally constant fibration with respect to the pair} $(X,\Delta)$ if it satisfies the following conditions:
	\begin{itemize}
	\item The morphism $\phi:X\to Y$  is a (locally trivial) analytic fibre bundle with the fibre $F$.
	\item Every component $\Delta_i$ of $\Delta$ is horizontal (i.e., $\phi(\Delta_i)=Y$).
	\item There exist a representation
	$$\rho:\pi_1(Y)\to \textup{Aut}(F)$$
	of the fundamental group $\pi_1(Y)$ of $Y$ to the automorphism group $\Aut(F)$ of $F$ and a Weil $\mathbb{Q}$-divisor $\Delta_F$ on $F$ which is invariant under the action of $\pi_1(Y)$, so that $(X,\Delta)$ is isomorphic to the quotient of $(Y^{\text{univ}}\times F, \textup{pr}_2^*\Delta_F)$ over $Y$ by the action of $\pi_1(Y)$ on $Y^{\textup{univ}}\times F$.
	Here $Y^{\textup{univ}}$ is the universal cover of $Y$, and the action of $\gamma\in\pi_1(Y)$ is defined to be
	$$\gamma\cdot(y,z):=(\gamma\cdot y,\rho(\gamma)(z))$$
	for any $(y,z)\in Y^{\textup{univ}}\times F$.
	\end{itemize}
 \item[(2)] The morphism $\phi:X\to Y$ is simply said to be \textit{locally constant} if it satisfies the above conditions for $\Delta=0$.
 \item[(3)] For the convenience of the notation, we denote by $\text{Aut}(F,\Delta_F)$  the group of  automorphisms of $F$ which leave  $\Delta_F$ invariant, counted with the multiplicities.
\end{enumerate}
\end{defn}

One of the most important examples of locally constant families  are   \emph{flat vector bundles}, which plays an important role in this paper.

\begin{defn}
Let $E\rightarrow Y$ be a holomorphic vector bundle of rank $r$ over a normal variety. Then we say that $E$ is \textit{flat} if $E\rightarrow Y$ is a locally constant family corresponding to a representation $\rho:\pi_1(Y)\rightarrow \text{GL}(r,\mathbb{C})$.
\end{defn}

Let $E\rightarrow Y$ be a holomorphic vector bundle over a normal projective variety.
Recall that $E$ is called \emph{nef} if the tautological line bundle $\scrO_{\PP E}(1)$ over $\PP E$ is nef, and $E$ is called \emph{numerically flat} if both $E$ and its dual bundle $E^*$ are nef. According to \cite[Theorem 1.18 and Corollay 1.19]{DPS94} and \cite[Section 3]{Sim92},  a numerically flat vector bundle $E\rightarrow Y$ over a normal projective variety is a flat vector bundle.

\begin{lemme}[{\cite[Lemma 4.4]{LOY19}}]
\label{lemma_fixed-pt-subbundle}
Let $V$ be a finite dimensional representation of $G:=\pi_1(Y)$ with $E$ being the flat vector bundle on $Y$ induced by this representation. 
Then there is a one-to-one correspondence between the set of $G$-fixed points $z\in\PP V$ and the set of codimension one flat subbundles of $E$.
\end{lemme}

This is nothing but a tautology fact (functorial definition of projective spaces); see \cite[Lemma 4.4]{LOY19} for the case when $Y$ is a smooth projective variety. 
Indeed, the projectivity and smoothness of  $Y$ are not needed in the proof. 
For readers' convenience, we include a detailed proof.

\begin{proof}[Proof of {\hyperref[lemma_fixed-pt-subbundle]{Lemma \ref*{lemma_fixed-pt-subbundle}}}]
Clearly, there is a one-to-one correspondence between the set of $G$-fixed-points $z\in\PP V$ (which corresponds to a hyperplane $V_z\subset V$) and the set of quotient representations $V\to V/V_z$.
By the bijection between the isomorphism classes of  local systems on $Y$ and isomorphism classes of  representations of the fundamental group $\pi_1(Y)$, there is a one-to-one correspondence between  the set of $G$-fixed-points $z\in\PP V$ and line bundle quotients $E\to F$.
Since the kernel of a surjective morphism between locally free sheaves is locally free, our lemma follows.
\end{proof}

\begin{lemme}[{\cite[Lemma 4.5]{LOY19}}]
	\label{l.fixedpointflatsection}
	Let $Y$ be a normal projective variety and $\pi:Y\univ\rightarrow Y$  the universal cover with Galois group $G=\pi_1(Y)$. 
	Let $\rho:G\rightarrow \text{GL}(V,\CC)$ be a linear representation of $G$ and denote by $E$ the corresponding flat vector bundle over $Y$. 
	If $z\in\PP(V)$ is a $G$-fixed-point, then it induces a section $\sigma:Y\rightarrow \PP E$   such that $\sigma^*\scrO_{\PP E}(1)$ is numerically trivial. 
\end{lemme}

\begin{proof}
Let $F$ be  the co-rank one flat  subbundle of 
$E$ corresponding to $z$ (cf.~{\hyperref[lemma_fixed-pt-subbundle]{Lemma \ref*{lemma_fixed-pt-subbundle}}}).  
Let $\sigma:Y\rightarrow \PP E$ be the section defined by the following short exact sequence
	\[
	0\rightarrow F\rightarrow E\rightarrow E/F\rightarrow 0.
	\]
	Then we get $\sigma^*\scrO_{\PP E}(1)\cong E/F$. 
	On the other hand, as both $F$ and $E$ are numerically flat, it follows that $E/F$ and hence $\sigma^*\scrO_{\PP E}(1)$ are also numerically flat.
In particular, the line bundle $\sigma^*\scrO_{\PP E}(1)$ is numerically trivial.
\end{proof}

\begin{rmq}
	In the above lemma, the image $\sigma(Y)$ is just the image of $Y\univ\times \{z\}$ under the natural morphism $Y\univ\times \PP(V)\rightarrow \PP E$. 
	Here, we identify $\PP E$ with  the quotient of $Y\univ\times \PP V$ by the induced diagonal action of $G$.
\end{rmq}

Let $f:X\rightarrow Y$ be a locally constant fibration with respect to the pair $(X,\Delta)$, and let $\mu:Y'\rightarrow Y$ be a morphism from a normal variety $Y'$. 
Then there exists a locally constant fibration $f':X'\rightarrow Y'$ with respect to a pair $(X',\Delta')$ induced by $f$. Indeed, we consider the natural composition
\[
\pi_1(Y')\xrightarrow{\mu_*} \pi_1(Y) \xrightarrow{\rho} \Aut(F,\Delta_F), 
\]
where $\rho$ is the representation associated to $f$. Then   $X' = X\times_{Y} Y'$.

With this base change kept in mind, we prove the following lemma.
\begin{lemme}
	\label{l.basechange}
	Let $f:X\rightarrow Y$ be a locally constant fibration between normal projective varieties with respect to a pair $(X,\Delta)$. Assume that $K_X+\Delta$ is $\QQ$-Cartier and $Y$ is $\QQ$-Gorenstein. Let $\mu:Y'\rightarrow Y$ be a morphism from a normal projective variety $Y'$ with $\QQ$-Gorenstein singularities. 
	Denote by $f':X'\rightarrow Y'$ the induced locally constant fibration by $f$ and let $g:X'\rightarrow X$ be the natural   morphism. 
	Then we have
	\[
	g^*(K_{X/Y}+\Delta) \sim_{\QQ} K_{X'/Y'}+\Delta'.
	\]
\end{lemme}

\begin{proof}
	Let $F$ be a general fibre of $f$ and let $\Delta_F$ be the $\QQ$-Weil divisor on $F$ as  in {\hyperref[defn-locally-constant]{Definition  \ref*{defn-locally-constant}}}. 
	Let $Y\univ$ and $(Y')\univ$ be the universal coverings of $Y$ and $Y'$, respectively. Then the morphism $\mu:Y'\rightarrow Y$ can be lifted to a morphism $\widetilde{\mu}:(Y')\univ\rightarrow Y\univ$ satisfying the following commutative diagram
	\[
	\begin{tikzcd}[column sep=large, row sep=large]
		(Y')\univ\times F \arrow[r,"\widetilde{g}"] \arrow[d,"\pi'" left]  
		    & Y\univ \times F  \arrow[d,"\pi"]  \\
		X' \arrow[r,"g" below]
		    & X
	\end{tikzcd}
	\]
	where $\widetilde{g}:=(\widetilde{\mu},\textup{id})$. Since $K_X+\Delta$ is $\QQ$-Cartier, the $\QQ$-Weil divisor $K_F+\Delta_F$ is also $\QQ$-Cartier. Moreover, since $\pi'$ and $\pi$ are both \'etale, we get
	\begin{center}
		$(\pi')^*(K_{X'/Y'}+\Delta')=(p_2')^*(K_F+\Delta_F)$ \quad and \quad $\pi^*(K_{X/Y}+\Delta)=p_2^*(K_F+\Delta_F)$,
	\end{center}
    where $p_2'$ and $p_2$ are the projections $(Y')\univ\times F\rightarrow F$ and $Y\univ\times F\rightarrow F$, respectively. 
    In particular, we obtain
    \[
    (g\circ \pi')^*(K_{X/Y}+\Delta) \sim (\pi\circ \widetilde{g})^*(K_{X/Y}+\Delta) \sim (\pi')^*(K_{X'/Y'}+\Delta').
    \]
    Note that, by our assumption, the line bundle $(p'_2)^*\scrO_{F}(mK_F+m\Delta_F)$ is equivariant with respect to the induced action of $\pi_1(Y')$ on $(Y')\univ\times F$ for 
     for $m$ sufficiently divisible. Thus we have $g^*(K_{X/Y}+\Delta)\sim K_{X'/Y'}+\Delta'$. 
\end{proof}

\vskip 2\baselineskip

\section{Periodic points of a special K\"ahler group}\label{sec-period}
In this section, we study the periodic points of an automorphism group, as a linear quotient of some special K\"ahler group, with 
the main results {\hyperref[prop-sol-fixed]{Proposition    \ref*{prop-sol-fixed}}} and {\hyperref[prop_fund_pp]{Proposition     \ref*{prop_fund_pp}}}.
As an application, we apply {\hyperref[prop_fund_pp]{Proposition     \ref*{prop_fund_pp}}}  to the locally constant MRC fibration $f:X\to Y$ in \cite[Theorem 1.2]{MW21}, which plays a significant role to find a section of such $f$  (cf.~{\hyperref[cor-lc-periodic]{Corollary \ref*{cor-lc-periodic}}} and  {\hyperref[lemma_fixed-pt-subbundle]{Lemma   \ref*{lemma_fixed-pt-subbundle}}}). 
We refer readers to \cite[Section 4]{LOY19} which considered the Albanese map onto the complex torus; however, the fundamental group of our $Y$ (as the object of the MRC fibration) here is neither abelian nor solvable.

To begin with this section,  we briefly recall some basic facts about the Shafarevich map, which will be used in the proof of this and also later sections. 
\begin{defn}[{\cite[Definition 3.5]{Kollar95}}]\label{defn-shafarevich}
Let $Y$ be a normal variety and $H\lhd \pi_1(Y)$ a normal subgroup of the fundamental group of $Y$.
A normal variety $\textup{Sh}^H(Y)$ and a rational map $\textup{sh}_Y^H:Y\dashrightarrow\textup{Sh}^H(Y)$ are called the \textit{$H$-Shafarevich variety} and the \textit{$H$-Shafarevich map} of $Y$ if
\begin{enumerate}
    \item[(1)] $\textup{sh}_Y^H$ has connected fibers, and
    \item[(2)] there are countably many closed subvarieties $D_i\subseteq Y$ $(D_i\neq Y)$ such that for every closed, irreducible subvariety $Z\subseteq Y$ with $Z\not\subseteq\cup D_i$, we have 
    $$\textup{sh}_Y^H(\widetilde{Z})=\textup{point}~~\textup{if and only if}~~\textup{im}[\pi_1(\widetilde{Z})\to\pi_1(X)]\textup{~has finite index in }H,$$
where $\widetilde{Z}$ is the normalization of $Z$.
\end{enumerate}
\end{defn}

The following proposition guarantees the existence of the Sharfarevich maps.
\begin{prop}[{cf.~\cite[Theorem 3.6]{Kollar95}}]\label{prop-kol-sha}
Let $X$ be a normal variety, and $H\lhd\pi_1(X)$ a normal subgroup. Then
\begin{enumerate}
\item[(1)]The $H$-Shafarevich map $\textup{sh}_X^H:X\dashrightarrow\textup{Sh}^H(X)$ exists.
\item[(2)]For every choice of $\textup{Sh}^H(X)$ (within its birational equivalence class) there are open subsets $X_\circ\subseteq X$ and $W_\circ\subseteq\textup{Sh}^H(X)$ with the following properties:
\begin{itemize}
    \item $\textup{sh}_X^H:X_\circ\to W_\circ$ is everywhere defined on $X_\circ$.
    \item Every fibre of $\textup{sh}_X^H|_{X_\circ}$ is closed in $X$.
    \item $\textup{sh}_X^H|_{X_\circ}$ is a topologically locally trivial fibration.
\end{itemize}
\item[(3)] If $X$ is proper, then $\textup{sh}_X^H:X_\circ\to W_\circ$ is proper and the very general points of $X$ is contained in $X_\circ$ for a suitable choice of $\textup{Sh}^H(X)$ and $X_\circ$.
\end{enumerate}
\end{prop}

Now, we apply {\hyperref[defn-shafarevich]{Definition   \ref*{defn-shafarevich}}}  and {\hyperref[prop-kol-sha]{Proposition  \ref*{prop-kol-sha}}} to our specific setting. 
Let $Y$ be a normal projective variety and let $\chi:\pi_1(Y)\rightarrow \text{GL}(r,\mathbb{C})$ be a representation. 
%By Selberg's lemma (\cite{Sel60}), every finitely generated linear group over $\CC$ is virtually torsion-free, and consequently there exists a finite \'etale cover 
Let $\pi:Y'\rightarrow Y$ be a finite \'etale cover such that the algebraic Zariski closure $\overline{G_{Y'}}$ of the image $G_{Y'}:=\chi(\pi_*(\pi_1(Y')))$ of the following composed map
\[
\chi':\pi_1(Y') \xrightarrow{\pi_*} \pi_1(Y) \xrightarrow{\chi} \text{GL}(r,\CC)
\]
is connected. 
Moreover, denote by $\text{Rad}(\overline{G_{Y'}})\lhd \overline{G_{Y'}}$ for the solvable radical of $\overline{G_{Y'}}$. 
Then the quotient $\overline{G_{Y'}}/\text{Rad}(\overline{G_{Y'}})$ is torsion-free. 
Denote by $K\lhd \pi_1(Y')$ the kernel of the map $\pi_1(Y')\rightarrow \overline{G_{Y'}}/\text{Rad}(\overline{G_{Y'}})$, which is normal in $\pi_1(Y')$.  
Then it follows from {\hyperref[prop-kol-sha]{Proposition  \ref*{prop-kol-sha}}} that there exists a dominant almost holomorphic map (which is the $K$-Shafarevich map) $\text{sh}_{Y'}^K:Y'\dashrightarrow \text{Sh}^K(Y')$ to a smooth projective variety. 
In other words, there exists a dense Zariski open subset $Y'_{\circ}\subset Y'$ such that the restriction $\text{sh}^K_{Y'}|_{Y'_{\circ}}$ is well-defined and proper. %This map $\text{sh}_{Y'}^K$ is called the \emph{$K$-Shafarevich map} of $Y'$ (cf.~\cite[Definition 3.5]{Kollar95}). In particular, we borrow the following result which sits in a more general setting:
Take  a very general point $x\in Y'_{\circ}$ and let $Z$ be a subvariety through $x$, with normalisation $n:\widetilde{Z}\rightarrow Z$.   
Then, the rational map $\text{sh}^K_{Y'}$ maps $Z$ to a point if and only if the composed map
\[
\pi_1(\widetilde{Z})\xrightarrow{n_*} \pi_1(Y') \rightarrow \overline{G_{Y'}}\rightarrow \overline{G_{Y'}}/\text{Rad}(\overline{G_{Y'}})
\]
has finite image (cf.~{\hyperref[defn-shafarevich]{Definition   \ref*{defn-shafarevich}}}).  
We refer readers to \cite{CCE15} for more information on the Shafarviech maps and Shafarviech varieties. 

The following theorem is crucial for our proof of {\hyperref[t.mainthm]{Theorem \ref*{t.mainthm}}}.

\begin{thm}[{\cite[Theorem 1]{CCE15}}]
	\label{t.basemfdShafarevichmap}
With the notation above, assume that $Y'$ has at worst klt singularities. Then the smooth projective variety $\text{Sh}^K(Y')$ is of general type.
\end{thm}

\begin{proof}
	Let $\eta:\widetilde{Y}'\rightarrow Y'$ be a resolution. By \cite[Theorem 1.1]{Tak03}, the push-forward $\eta_*:\pi_1(\widetilde{Y}')\rightarrow \pi_1(Y')$ is an isomorphism. In particular, the composed map 
	\[
	\widetilde{Y}' \rightarrow Y' \dashrightarrow \text{Sh}^K(Y')
	\]
	is a Shafarevich map for $K\lhd \pi_1(\widetilde{Y}')$. Then the statement follows immediately from \cite[Th\'eor\`eme 1]{CCE15} and the discussion above.
\end{proof}

In the sequel of this section, we study the periodic point of certain groups acting on a projective varieties. By the Borel fixed point theorem, we generalize \cite[Theorem 4.1]{LOY19} to the solvable group case to get the first main result of this section.
Compared with \cite[Theorem 4.1]{LOY19}, our argument is in the line of the theory of algebraic group.

Let $\sigma:G\times F\to F$ be a group action of $G$ on a projective variety $F$.
We say that $G$ has a \textit{fixed point} $y\in F$, if  $\sigma(g,y)=y$ for any $g\in G$.
We say that $G$ has a \textit{periodic point} $y\in F$, if there exists a positive integer $m$ such that $\sigma(g^m,y)=y$ for any $g\in G$. 

\begin{prop}\label{prop-sol-fixed}
	Let $G$ be a group acting on a projective variety $F$ via a group homomorphism $\rho: G\to \textup{Aut}(F)$ to the automorphism group of $F$ such that the image $\rho(G)$ is virtually solvable and there is a $G$-linearized ample line bundle $L$ on $F$.
	Then $G$ has a periodic point on $F$.	
\end{prop}
\begin{proof}
	After replacing  $G$ by a finite index subgroup, we can assume that $\Gamma:=\rho(G)$ is solvable. By hypothesis, $\rho(G)$ is a subgroup of the linear algebraic group $\textup{Aut}_L(F)$.
	Denote by $[\Gamma,\Gamma]$ the commutator subgroup of $\Gamma$.  
	Recall that for an algebraic group $H$, its commutator group $[H,H]$ is a closed subgroup of $H$, hence also algebraic by \cite[Chapter I, \S 2.3 Proposition, pp.~58-59]{Bor69}.
	Hence,  %$\overline{[\Gamma,\Gamma]}\subseteq [\overline{\Gamma},\overline{\Gamma}]$ and thus
	we have $\overline{[\Gamma,\Gamma]}=[\overline{\Gamma},\overline{\Gamma}]$ by \cite[Chapter I, \S2.1(e), p.~57]{Bor69}.	
	Inductively, the closure of the $p$-th derived group  $\overline{\Gamma^{(p)}}$ coincides with the $p$-th derived group $\overline{\Gamma}^{(p)}$ of $\overline{\Gamma}$.
	
	By the solvability of $\Gamma$,  the $p$-th derived group $\Gamma^{(p)}$ is trivial for some $p<\infty$ (i.e., the derived length of $\Gamma$ is finite).
	Hence, $\overline{\Gamma}^{(p)}$ is also trivial for some $p<\infty$ and thus $\overline{\Gamma}$ is a solvable algebraic group.
	With 
	%$\overline{\Gamma}$ 
	$G$ replaced by some subgroup of finite index if necessary, we can  assume that $\overline{\Gamma}$ is connected.
	By Borel's fixed-point theorem (cf.~\cite[Theorem 10.4, p.~137]{Bor69}), the action of $\overline{\Gamma}$ on $F$ has a fixed point and our proposition is proved by considering the exact sequence $1\to\ker\to G\to G|_F\to 1$.
\end{proof}

In what follows, we recall the notion of {\it special varieties} in the sense of Campana. 
Roughly speaking, a special variety is a compact complex variety in the Fujiki class $\mathcal{C}$ (i.e., bimeromorphic to a compact K\"ahler manifold) that admits no orbifold-theoretic (in the sense of F. Campana) meromorphic fibration onto a positive-dimensional variety of (orbifold) general type. For the precise definition, see \cite[Definition 2.1(2)]{Cam04}. 

%we refer readers to \cite[\S 2]{Cam04} for the definition of \textit{special} (analytic) varieties in the sense of Campana. 
Note that there are two fundamental examples of special varieties (\cite[Theorem 3.22, Theorem 5.1]{Cam04}) as indicated in the following lemma; in fact, every special variety is essentially built up from them (\cite[Example 2.3, \S 6.5]{Cam04}).

\begin{lemme}\label{lem-kappa-special}
Let $X$ be 
a compact complex variety in the Fujiki class $\mathcal{C}$.
%a compact K\"ahler manifold. 
Then $X$ is special if one of the following holds:
\begin{itemize}
\item[\rm(1)] $X$ is rationally connected;
\item[\rm(2)] The Kodaira dimension vanishes, i.e., $\kappa(X)=0$.	
\end{itemize}
\end{lemme}

\begin{rmq}
One should notice that the notion `rational connectedness' in \cite{Cam04} (cf. \cite[\S 3.3]{Cam04}) is usually called `rational chain connectedness' (cf. \cite[3.2 Definition]{Kollar96}), and this is why in \cite[Theorem 3.22]{Cam04} the smoothness of $X$ is required. In this article, we take the usual definition of `rational connectedness' (i.e., any two general points can be connected by a rational curve), and since this is by definition a bimeromorphic property for complex varieties, we do not need the smoothness condition in the lemma. Let us remark that `rational chain connectedness' is however not a bimeromorphic property, this is why the lemma does not hold in this generality; yet the two notions `rational connectedness' and `rational chain connectedness' coincide for varieties with dlt singularities by \cite[Corollary 1.5(2)]{HM07}.      
\end{rmq}

Besides, special varieties have the following property, known as the `weak speciality' (cf. \cite[\S 9]{Cam04}):
%Besides, the following lemma is known to experts.
\begin{lemme}[{cf.~\cite[Proposition 9.27]{Cam04}}]\label{lem-special-no-to-gt}
Let $X$ be a special variety. 
%lying in the Fujiki class $\mathcal{C}$ (i.e., a special variety bimeromorphic to a compact K\"ahler manifold). 
Then any finite \'etale cover of $X$ can never dominate a positive dimensional variety of general type.
\end{lemme}

%The above property is also said to be \textit{weak speciality} (cf.~\cite[Section 9]{Cam04}).  
In the sequel we will prove the main results of this section concerning the special K\"ahler groups. First recall that a group $G$ is said to be a \textit{K\"ahler group} if $G$ can be realized as the fundamental group of a compact K\"ahler manifold. 
Further, a K\"ahler group $G$ is said to be \textit{special}, if $G$ can be realized as the fundamental group of a special compact K\"ahler manifold.

Applying \cite{Tak03} to the resolution of some projective klt pair, we immediately have the following lemma.
\begin{lemme}\label{lem-special-klt}
Let $(X,\Delta)$ be a projective klt pair.
If $X$ is special, then the fundamental group $\pi_1(X)$ is a special K\"ahler group.
\end{lemme}

\begin{prop}\label{prop-special-Kahler-abe}
Any linear quotient of a special K\"ahler group is virtually abelian.
\end{prop}

%Before we prove {\hyperref[prop-special-Kahler-abe]{Proposition   \ref*{prop-special-Kahler-abe}}}, we refer readers to 

%Now we come back to the proof of {\hyperref[prop-special-Kahler-abe]{Proposition   \ref*{prop-special-Kahler-abe}}}. 
\begin{proof}%[Proof of {\hyperref[prop-special-Kahler-abe]{Proposition   \ref*{prop-special-Kahler-abe}}}]
Let $X$ be a compact K\"ahler manifold and $G:=\pi_1(X)$ the fundamental group.
Let $\chi:G\to \textup{GL}(r,\mathbb{C})$ be a linear representation.
Let $G_X:=\chi(G)$ be the image and $\overline{G_X}$ its Zariski closure. 
We will prove that $G_X$ is virtually abelian. To this end, we can freely replace $G$ (and thus $G_X$ and $\overline{G_X}$) by a finite index subgroup, and correspondingly replace $X$ by a finite \'etale cover, 
%After replacing $X$ by a finite \'etale cover, 
so that we may assume that $\overline{G_X}$ is a connected algebraic group.
Denote by $\overline{R}$ the solvable radical of $\overline{G_X}$ and let $R:=\overline{R}\cap G_X$.
Let $s:G_X\to G_X/R$ be its natural quotient.

By \cite[Th\'eor\`eme 6.5 and its proof]{CCE15}, after replacing $X$ by a further  \'etale cover, there is a proper modification $X'\to X$ such that $X'$ dominates a variety  which is bimeromorphic to the total space of a smooth fibration $\textup{Sh}^{K_1}(X)\to W$ over the Shafarviech variety $\textup{Sh}^{K_2}(X)$ such that  $W$ is of general type, where $K_1$ is the kernel of $\chi$ and $K_2$ is the kernel of $s\circ\chi$. 
Since $X$ is special, our $W$ has to be a single point (cf.~{\hyperref[lem-special-no-to-gt]{Lemma  \ref*{lem-special-no-to-gt}}}).
This in turn implies that the Shafarviech variety $\textup{Sh}^{K_2}(X)$  is a single point and thus the image $s(G_X)=s\circ\chi(\pi_1(X))$ is finite.
Hence, $G_X$ itself is virtually solvable (also cf.~\cite[Th\'eor\`eme 6.3]{CCE15}).
Now, it follows from \cite[Corollaire 4.2]{CCE15} that $G_X$ is virtually abelian, which gives our proposition.
\end{proof}

The following corollary is an immediately consequence of {\hyperref[prop-special-Kahler-abe]{Proposition  \ref*{prop-special-Kahler-abe}}}. 
%We refer readers to \cite[Section 11]{GGK19} for the linear representations of the fundamental group of a klt projective variety with the numerical trivial canonical divisor and the vanishing augmented irregularity.

\begin{cor}\label{cor-virtually-abelian}
Let $X$ be a normal projective variety with at worst klt singularities.
Suppose that $X$ is special in the sense of Campana (this is the case when $\kappa(X)=0$; cf.~{\hyperref[lem-kappa-special]{Lemma  \ref*{lem-kappa-special} (2)}}).
Then any linear quotient of $\pi_1(X)$ is virtually abelian.
\end{cor}

\begin{proof}
By {\hyperref[lem-special-klt]{Lemma \ref*{lem-special-klt}}}, 
$\pi_1(X)$ is a special K\"ahler group (cf.~\cite{Tak03}),
noting that the specialness is a birational invariant. 
Then our corollary follows from {\hyperref[prop-special-Kahler-abe]{Proposition  \ref*{prop-special-Kahler-abe}}}. 
\end{proof}

Now we are in the position to prove the second main result in this section.

\begin{prop}\label{prop_fund_pp}
Let $G$ be a special K\"ahler group.
Suppose that  $G$ acts on a projective variety $F$ via a group homomorphism $\rho:G\to\Aut(F)$ to the automorphism group of $F$, such that there exists a $G$-linearized  ample line bundle $L$ on $F$  (cf. \cite[Defintion 3.2.3, Lemma 3.2.4]{Brion18}).
Then $G$ has a periodic point on $F$.
\end{prop}

\begin{proof}
	It follows directly from {\hyperref[prop-sol-fixed]{Proposition \ref*{prop-sol-fixed}}} and {\hyperref[prop-special-Kahler-abe]{Proposition \ref*{prop-special-Kahler-abe}}}.
\end{proof}

Note that the solvability of special K\"ahler groups has its own  interests (cf.~\cite[Conjecture 7.1]{Cam04}). Thus we may ask the following question.

\begin{ques}\label{ques-kernel-sol}
	Let $G$ be a special K\"ahler group.
	Suppose that $G$ acts on a projective variety $F$ via a group homomorphism $G\to\Aut(F)$ such that there is a $G$-linearized  ample line bundle $L$ on $F$.
	Can we give some descriptions on the kernel of $\rho: G\to G|_F$? 
	Will $\ker\rho$ (and hence $G$) be virtually solvable?
\end{ques}

In our case, the solvability of $G$ follows from the solvability of $\ker\rho$ due to the exact sequence $1\to \ker\rho\to G\to G|_F\to 1$  (cf.~{\hyperref[prop-special-Kahler-abe]{Proposition  \ref*{prop-special-Kahler-abe}}}). 
Moreover, {\hyperref[ques-kernel-sol]{Question   \ref*{ques-kernel-sol}}} has a negative answer if we don't assume
the speciality.  
For example, 
it is known that the group given by the presentation
$$\Gamma_g=\left<\alpha_1,\cdots,\alpha_g,\beta_1,\cdots,\beta_g:\prod_{i=1}^g[\alpha_i,\beta_i]=1\right>$$
is K\"ahler.
Indeed, it is the fundamental group of a compact Riemann surface of genus $g$.
However, when $g>1$ the commutator subgroup $[\Gamma_g,\Gamma_g]$ is a free non-abelian group and hence $\Gamma_g$ is not virtually solvable. 
%(cf.~https://arxiv.org/pdf/2101.05905.pdf, Theorem B)?

%\begin{rmq}
%Indeed, the existence of the ample $G$-invariant line bundle has its own interest. 
%In this situation, $G|_F$ would be virtually contained in the identity component $\textup{Aut}_0(F)$ of $\textup{Aut}(F)$ by Fujiki and Lieberman (cf.~\cite{Fuj78} and \cite{Lie78}).
%However, little things were know whether $G$ is solvable or not (cf.~Question \ref{ques-kernel-sol}).
%\end{rmq}

Applying {\hyperref[prop_fund_pp]{Proposition  \ref*{prop_fund_pp}}} to a locally constant fibration, we end up this section with the following corollary. % which will be used in the proof of {\hyperref[mainthm-section]{Theorem \ref*{mainthm-section}}}.
\begin{cor}\label{cor-lc-periodic}
Let $f:X\to Y$ be a locally constant fibration with a fibre $F$.
Suppose that the irregularity $q(F)=0$ and $Y$ is special.
Then the fundamental group $G:=\pi_1(Y)$ acts on $F$ and has a periodic point on it.
\end{cor}

\begin{proof}
To apply {\hyperref[prop_fund_pp]{Proposition  \ref*{prop_fund_pp}}}, we only need to verify that there is a $G$-linearized ample line bundle $L$ on $F$ (cf.~\cite[Definition 2.6]{MW21}). 
Since $f$ is locally constant, we have the following diagram
\[
\xymatrix{F\times Y^{\textup{univ}}\ar[r]\ar[d]&Y^{\text{univ}}\ar[d]\\
X\ar[r]&Y}
\]
where $Y^{\textup{univ}}$ is the universal cover of $Y$ and $F\times Y^{\textup{univ}}\cong X\times_Y Y^{\textup{univ}}$ with the induced projection $\textup{pr}_1: F\times Y^{\textup{univ}}\to F$ and  $\textup{pr}_2: F\times Y^{\textup{univ}}\to Y^{\textup{univ}}$.
Let us fix an $f$-ample divisor $A$ on $X$ and denote by $p_X: F\times Y^{\textup{univ}}\to X$ the natural projection.
Applying \cite[Proof of Lemma 2.7]{MW21}, we see that there exists a line bundle $A_Y$ on $Y$ such that 
$$p_X^*(A-f^*A_Y)\cong\textup{pr}_1^*A_F$$ 
for some line bundle $A_F$ on $F$.
Since $X\cong (F\times Y^{\textup{univ}})/G$, our $p_X^*(A-f^*A_Y)$ is $G$-linearized, and hence,  
$$g^*p_X^*(A-f^*A_Y)\sim p_X^*(A-f^*A_Y)$$ 
for any $g\in G$ (regarded as an automorphism of $F$); see \cite[Defintion 3.2.3, Lemma 3.2.4]{Brion18}.  
Hence, $\textup{pr}_1^*A_F$ is also $G$-linearized.
Moreover, since $G$ acts on $F\times Y^{\textup{univ}}$ via the diagonal, our $A_F$ is a $G$-linearized ample line bundle, noting that $A_F\cong A|_F$.
\end{proof}

\vskip 2\baselineskip

%{\color{blue} The following is a discussion on $G_1$, and can be removed finally.}
%Let $X\to Y$ be a locally constant fibration. 
%We denote by $G_1:=\{g\in G~|~g|_F=\text{id}\}$ and consider the following diagram:
%\[
%\xymatrix{F\times Y^{\textup{univ}}\ar[r]\ar[d]&Y^{\text{univ}}\ar[d]\\
%F\times (Y^{\text{univ}}/G_1)\ar[r]\ar[d]&Y^{\text{univ}}/G_1\ar[d]\\
%X\ar[r]&Y}
%\]

%From \cite[Theorem 6.5]{CCE15}, we know that:
%Let $Y$ be a compact K\"ahler manifold and $\rho:\pi_1(Y)\to \textup{GL}_N(\mathbb{C})$ a linear representation of its fundamental group.
%Then after replacing $Y$ by a finite \'etale cover, the Shafarevich manifold $\textup{Sh}_\rho(Z)$ associated with $\rho$ is (bimeromorphic to) a compact K\"ahler manifold $Z$ which is the total space of a smooth torus fibration over a manifold of general type.
%$$Y\xrightarrow{\textup{sh}_\rho}Z:=\textup{Sh}_\rho(Y)\xrightarrow{s_\rho}S_\rho(Y)$$
%The manifold $Z$ is bimeromorphic to the Shafarevich manifold of $Y$, obtained from $Y$ by contracting the submanifolds, the fundamental groups of which lie in the kernel of $\rho$.
%The torus fibration (i.e., the submersion $Z\to S_\rho(Y)$) has the tori fibre being the maximal manifolds, on the fundamental groups of which, $\rho$ has an abelian image in $\textup{GL}_N(\mathbb{C})$.
%Note that $\rho$ factors as $\pi_1(Y)\to\pi_1(Z)\to \rho(\pi_1(Y))$ (cf.~\cite[Section 2]{CCE15}).

%\section{Sections of MRC fibrations, Proof of \texorpdfstring{{\hyperref[mainthm-RC]{Theorem \ref*{mainthm-RC}}}}{text}}
\section{Proofs of \texorpdfstring{{\hyperref[mainthm-RC]{Theorem \ref*{mainthm-RC}}}}{text},  \texorpdfstring{{\hyperref[t.mainthm]{Theorem \ref*{t.mainthm}}}}{text} and \texorpdfstring{{\hyperref[c.algebraicallyintegrable]{Corollary \ref*{c.algebraicallyintegrable}}}}{text}}

In this section, we shall prove our main results {\hyperref[mainthm-RC]{Theorem \ref*{mainthm-RC}}} and {\hyperref[t.mainthm]{Theorem \ref*{t.mainthm}}}. 
As an application of {\hyperref[t.mainthm]{Theorem \ref*{t.mainthm}}}, we show {\hyperref[c.algebraicallyintegrable]{Corollary \ref*{c.algebraicallyintegrable}}}.

\begin{assumption}
\label{a.assumption}
Throughout this section, we follow the following assumption and its notations.
\begin{itemize}
    \item Let $(X,\Delta)$ be a projective klt pair, and  $f:X\rightarrow Y$ a locally constant fibration (cf.~{\hyperref[defn-locally-constant]{Definition \ref*{defn-locally-constant}}}) onto a $\mathbb{Q}$-Gorenstein normal projective variety $Y$ such that the relative canonical divisor $-(K_{X/Y}+\Delta)$ is nef. If $Y$ is smooth, then it follows from {\hyperref[t.structureantinef]{Theorem \ref*{t.structureantinef}}} below that $f$ is a locally constant fibration.
    \item Let $F$ be a fibre of $f$ and $G:=\pi_1(Y)$. 
    \item Assume that there exists a very ample  divisor $A_F$ on $F$ which is $G$-linearized (cf.~\cite[Definition 3.2.3, Lemma 3.2.4]{Brion18}). If $Y$ is smooth, then the existence of such $A_F$ follows from {\hyperref[t.structureantinef]{Theorem \ref*{t.structureantinef}}}. 
    \item By {\hyperref[defn-locally-constant]{Definition \ref*{defn-locally-constant}}}, our $G$ acts on  $F$ via the homomorphism $G\to \textup{Aut}_{A_F}(F)$.  
    \item Denote by $\rho: \pi_1(Y) \to \text{GL}(V,\CC)$ the representation induced by $ \pi_1(Y)\rightarrow \text{Aut}(F,\Delta_F)$, where $V=\Coh^0(F, A_F)$.
Denote by $Y^{\textup{univ}}$ the universal cover of $Y$. 
Then we have the following (Cartesian) commutative diagram:
\begin{center}
\begin{tikzpicture}[scale=2.5]
\node (A) at (0,0) {$Y$.};
\node (B) at (0,1) {$Y^{\textup{univ}}$};
\node (A1) at (-1.8,0) {$X$};
\node (B1) at (-1.8,1) {$X{\times}_YY\univ\simeq Y\univ\times F$};
\node (E) at (-0.9, 0.5) {$\square$};
\path[->,font=\scriptsize,>=angle 90]
(B) edge node[right]{$p_Y$} (A)
(B1) edge node[left]{$p_X$} (A1)
(B1) edge node[above]{$f\univ=\pr_1$} (B)
(A1) edge node[below]{$f$} (A);
\end{tikzpicture}
\end{center}
where $G$ acts on $Y\univ\times F$ via diagonal and $X\cong (Y\univ\times F)/G$.
Let $\pr_2:X{\times}_YY\univ\simeq Y\univ\times F\to F$ be the second  projection.
\end{itemize}
\end{assumption}

First recall the following structure theorem and criterion for locally constant fibrations. It is essentially proved in \cite{CCM19} and \cite{MW21} and it is the starting point for the proofs of our {\hyperref[mainthm-RC]{Theorem \ref*{mainthm-RC}}} and {\hyperref[t.mainthm]{Theorem \ref*{t.mainthm}}}.

\begin{thm}[{\cite{CCM19}, \cite[Theorem 4.7]{MW21}}]	\label{t.structureantinef}
Let $(X,\Delta)$ be a klt pair, and  $f:X\rightarrow Y$ a fibration to a smooth projective variety. Let $F$ be a general fibre of $f$. If the relative anti-log canonical divisor $-(K_{X/Y}+\Delta)$ is nef, then  $f$ is a locally constant fibration induced by a  representation $\pi_1(Y)\rightarrow \Aut(F,\Delta_F)$. Moreover, there exists a sufficiently ample line bundle $A$ on $F$  which is $\pi_1(Y)$-linearized (cf. \cite[Defintion 3.2.3, Lemma 3.2.4]{Brion18}). 
\end{thm}

%For the reader's convenience, let us briefly recall the proof. Let $H$ be an ample enough line bundle on $X$ such that the following maps
%\[
%\text{Sym}^m H^0(X_y,H|_{X_y}) \rightarrow H^0(X_y, H^m|_{X_y})
%\]
%are  surjective for all $m\in \bbN$ and $y\in Y$ general, where $X_y$ is the fibre of $f$ over $y$. Thanks to \cite{CaoHoering2019}, after replacing $H$ by its multiple, we may assume that the determinant of the direct image sheaf $\det(f_*H)$ is numerically equivalent to $D^r$ for some line bundle $D$ on $Y$, where $r$ is the rank of $f_*\sO_X(H)$. 
%Denote by $\widetilde{H}$ the line bundle $H \otimes (f^*D)^{-1}$. 
%Then it is proved in \cite[Section 3]{CampanaCaoMatsumura2019} that the direct image sheaves $E_m\coloneqq f_* (\widetilde{H}^m)$ and $V_{m,p}\coloneqq f_*(\sO_{X}(-p\Delta)\otimes \widetilde{H}^m)$ are locally free and numerically flat for any $m\in\bbZ$ and any $p\in \bbN$ sufficiently divisible. Denote by $E$ the numerical flat vector bundle $f_*\widetilde{H}$. Let $\pi:\widetilde{Y}\rightarrow Y$ be the universal cover. Then we have a natural isomorphism $\bbP(\pi^*E)\cong \widetilde{Y}\times \bbP^{r-1}$. Then the numerical flatness of $E_m$ and $V_{m,p}$ implies that the defining equations of the fibres of the natural relative embedding 
%\[
%(\widetilde{Y}\times_{Y} X, \widetilde{Y}\times_Y \Delta) \rightarrow \widetilde{Y}\times \bbP^{r-1}
%\]
%is independent of $y\in \widetilde{Y}$ (see for instance \cite[Proposition 2.8]{MatsumuraWang2021}).

The smoothness assumption of $Y$, however, cannot be removed in the theorem above: 
%{\hyperref[t.structureantinef]{Theorem \ref*{t.structureantinef}}}.
\begin{exa}
\label{e.counter-example-for-singlar-base}
Let $Y\subset \PP^{n+1}$ be a two-dimensional projective cone over a rational normal curve $C_n\subseteq \PP^n$. Then $Y$ is a Fano variety with klt singularities. Let $f:X\rightarrow Y$ be the blow-up of $Y$ at the vertex with exceptional divisor $E$. Then $X$ is smooth and $E$ is a smooth rational curve such that 
\[
f^*K_Y=K_X+\frac{1}{n}E.
\]
In particular, the pair $(X,\frac{1}{n}E)$ is klt and $-(K_{X/Y}+\frac{1}{n}E)$ is trivial and hence nef. In general, let $Y$ be a projective variety with klt singularities, but not terminal. Let $f:X\rightarrow Y$ be a terminal modification of $Y$. Let $\Delta$ be the $\QQ$-Weil divisor on $X$ defined as $K_X+\Delta=f^*K_Y$. Then $\Delta$ is effective, $(X,\Delta)$ is klt and the relative  anti-log canonical divisor $-(K_{X/Y}+\Delta)$ is trivial and hence nef. 
\end{exa}

In the following, we first show that if the fundamental group of a subvariety $Z$ of $Y$ is virtually solvable, then $-(K_{X/Y}+\Delta)$ cannot be strictly nef.

\begin{prop}
	\label{p.solvableflatsection}
	Under {\hyperref[a.assumption]{Assumption \ref*{a.assumption}}}, let $\eta:Z\rightarrow Y$ be a morphism from another  $\mathbb{Q}$-Gorenstein normal projective   variety $Z$. 
	If $\pi_1(Z)$ has a periodic point on $F$ via the composed map
	\[
	\bar{\rho}:\pi_1(Z) \xrightarrow{\eta_*} \pi_1(Y) \xrightarrow{\rho} \textup{GL}(V,\CC)
	\]
 then up to replacing $Z$ by some  finite \'etale cover if necessary,  
	there is a lifting $\sigma:Z\rightarrow X$ of $\eta:Z\rightarrow Y$ such that the pull-back $\sigma^*(K_{X/Y}+\Delta)$ is numerically trivial.
Further, if the image of $\pi_1(Z)$ under $\bar{\rho}$ is virtually solvable, then we can get such a periodic point on $F$. 
\end{prop}

Before we prove {\hyperref[p.solvableflatsection]{Proposition  \ref*{p.solvableflatsection}}}, we prepare the following lemma, which is an analogue of {\hyperref[l.fixedpointflatsection]{Lemma \ref*{l.fixedpointflatsection}}}. 

\begin{lemme}
\label{lem_lcf-section}
Under {\hyperref[a.assumption]{Assumption \ref*{a.assumption}}}, we have the following statements. 
\begin{itemize}
\item[\rm(a)] Every $G$-fixed point $z$ (if exists) induces a section $\sigma:Y\to X$ of the locally constant fibration $f: (X,\Delta)\to Y$.
\item[\rm(b)] Let $A_X$ be the quotient bundle of $A_F$ on $X$; that is $p_X^*A_X=\textup{pr}_2^*A_F$ (cf.~{\hyperref[a.assumption]{Assumption \ref*{a.assumption}}}).
Set $E:=f_\ast\scrO_X(A_X)$. 
Then every $G$-fixed point $z$ induces a short exact sequences of flat vector bundles 
\[
0\to E'\to E\to \scrO_Y(\sigma^\ast\!A)\to 0.
\]
In particular, $\sigma^*A_X\equiv 0$. 
\end{itemize}
\end{lemme}
\begin{proof}
Let $i:X\hookrightarrow\PP E$ be the embedding over $Y$ induced by $A_X$, with $\scrO_X(A_X)\simeq i^\ast\scrO_{\PP E}(1)$. 
For each fibre $F$, we get a $G$-equivariant embedding $i_F: F\hookrightarrow\PP(\Coh^0(F,\scrO_F(A_F)))$. 
By {\hyperref[l.fixedpointflatsection]{Lemma \ref*{l.fixedpointflatsection}}}, the $G$-fixed point $z\in F\subseteq\PP(\Coh^0(F, \scrO_F(A_F)))$ induces a subbundle $E'$ of $E$ so that $Q:=E/E'$ is a flat line bundle on $Y$ and the surjection $E\twoheadrightarrow Q$ corresponds to a section $\sigma_1: Y\to\PP E$ of the projective bundle $\PP E\to Y$  with $\sigma_1^\ast\scrO_{\PP E}(1)=Q\equiv 0$. 
But for every $y\in Y$, $\sigma_1(y)$ corresponds to $z\in F$ under the (canonical) identification $X_y\simeq F$. Hence, the image $\sigma_1(Y)$ is contained in the image of $X$ under the canonical embedding. 
%Since the image $\simga_1(Y)$ is just the image of $Y^{\text{univ}}\times \{z\}\subseteq$ under the natural morphism  $Y^{\text{univ}}\times\mathbb{P}(\Coh^0(F,\scrO_F(A_F)))\to\mathbb{P}(E)$ (cf.~).
Denote by $\sigma$ the induced morphism $Y\to X$  and we then have
\[
0\equiv Q\simeq\sigma_1^\ast\scrO_{\PP E}(1)\simeq\sigma^\ast\scrO_X(A)=\scrO_Y(\sigma^\ast A),
\]
noting that $\sigma_1=i\circ\sigma$.
%The second assertion of (b) follows from the fact that every numerically flat line bundle is numerically trivial.
\end{proof}

\begin{proof}[Proof of {\hyperref[p.solvableflatsection]{Proposition  \ref*{p.solvableflatsection}}}]
	Let $X'\rightarrow Z$ be the locally constant fibration induced by $f$ (cf.~the arguments before {\hyperref[l.basechange]{Lemma \ref*{l.basechange}}}). 
	By {\hyperref[l.basechange]{Lemma \ref*{l.basechange}}}, after replacing $Y$ by $Z$ and $X$ by $X'$, we may assume that $Z=Y$. 
	Denote by $\overline{G_Y}$ the Zariski closure of the image $G_Y:=\rho(\pi_1(Y))$ in $\text{GL}(V,\CC)$. 

First, we prove the second part of this proposition. 
Suppose that $G_Y$ is virtually solvable.
Then its closure $\overline{G_Y}$ is a virtually solvable linear algebraic group, which has finitely many components, and both $F$ and $\Delta_F$ are invariant under the induced action $\overline{G_Y}$ on $\PP(V)$. 
	In particular, after replacing $Y$ by an appropriate finite \'etale cover, we may assume that $\overline{G_Y}$ itself is connected and solvable. 
Thus, by Borel's fixed point theorem (cf.~\cite[Theorem 10.4, p.~137]{Bor69}), there exists a $\overline{G_Y}$-fixed-point $y\in F\subseteq  \PP(V)$ (cf.~{\hyperref[prop-sol-fixed]{Proposition  \ref*{prop-sol-fixed}}}), which completes the second part of our proof.
	
Now we prove the first part of this proposition. Let $m\in \NN$ be a sufficiently divisible positive integer. Then $-m(K_F+\Delta_F)+A_F$ is ample as $-(K_{F}+\Delta_F)=-(K_{X/Y}+\Delta)|_F$ is nef (cf.~{\hyperref[a.assumption]{Assumption \ref*{a.assumption}}}). Since both $K_{F}+\Delta_F$ and $A_F$ are $\overline{G_Y}$-linearized, for any positive integer $p$, the induced action of $\overline{G_Y}$ on the vector space
\[
V_{m,p}\coloneqq \Coh^0(F,\scrO_F(p(-m(K_F+\Delta_F)+A_F))
\]
induces a flat vector bundle structure on the direct image sheaf 
\[
E_{m,p}\coloneqq f_\ast\scrO_X(p(-m(K_{X/Y}+\Delta)+A_X)).
\]
Recall that $A_X$ is the quotient line bundle of $A_F$ on $X$, that is, we have $p_X^*A_X=\textup{pr}_2^*A_F$, where $p_X$ is the natural morphism and $p_2:Y^{\textup{univ}}\times F\rightarrow F$ is the second projection (cf.~{\hyperref[a.assumption]{Assumption \ref*{a.assumption}}}). 	
Moreover, by the ampleness of $-m(K_F+\Delta_F))+A_F$, for sufficiently large $p\gg 1$ there exists an embedding $F\hookrightarrow \PP(V_{m,p})$ such that the restriction of the action $\overline{G_Y}$ on $\PP(V_{m,p})$ to $F$ coincides with the restriction of the action $\overline{G_Y}$ on $\PP(V)$ to $F$. In particular, the point $y\in F\subseteq \PP(V_{m,p})$ is a $\overline{G_Y}$-fixed point for the action of $\overline{G_Y}$ on the projective space $\PP(V_{m,p})$. Denote by $\sigma:Y\rightarrow X$ the section induced by $y$. Applying  {\hyperref[lem_lcf-section]{Lemma \ref*{lem_lcf-section}}} to $E_{m,1}$ $(m\not=0)$ and $E_{0,1}$, we see that the pull-backs
\begin{center}
$\sigma^*(-m(K_{X/Y}+\Delta)+A_X)$ \quad \quad and \quad \quad $\sigma^*A_X$
\end{center}
are both numerically trivial. As a consequence, the pull-back $\sigma^*(K_{X/Y}+\Delta)$ is numerically trivial.
\end{proof}

As a consequence of {\hyperref[p.solvableflatsection]{Proposition  \ref*{p.solvableflatsection}}}, 
we are able to prove our first main result {\hyperref[mainthm-RC]{Theorem \ref*{mainthm-RC}}}. 
After that, we slightly generalize {\hyperref[mainthm-RC]{Theorem \ref*{mainthm-RC}}}  to the case when the anti-log canonical divisor is almost strictly nef (cf.~{\hyperref[prop-aug-irrg-bc]{Proposition \ref*{prop-aug-irrg-bc}}}). 
\begin{rmq}
In the spirit of \cite[Section 4]{LOY19}, the key step to show {\hyperref[mainthm-RC]{Theorem \ref*{mainthm-RC}}} is to find a $(K_X+\Delta)$-trivial section of the  maximal rationally connected fibration (MRC fibration for short), up to replacing $X$ by a quasi-\'etale cover, which is {\hyperref[p.solvableflatsection]{Proposition  \ref*{p.solvableflatsection}}}. 
In general, the MRC fibration for a projective variety may not be holomorphic.
However, by a recent joint work of Matsumura and the third author, it has been shown that, if $(X,\Delta)$ is a projective klt pair with the anti-log canonical divisor $-(K_X+\Delta)$ being nef, then up to replacing $X$ by its quasi-\'etale cover, there is a locally constant (holomorphic) fibration $f:X\to Y$ with respect to the pair $(X,\Delta)$ such that $K_Y\equiv 0$ (\cite[Theorem 1.1]{MW21}; cf.~\cite{Cao19,CH19,CCM19,Wang20}). 
\end{rmq}

\begin{proof}[Proof of {\hyperref[mainthm-RC]{Theorem \ref*{mainthm-RC}}}]
First, we note that $X$ is uniruled (cf.~{\hyperref[prop_strict_uniruled]{Proposition \ref*{prop_strict_uniruled}}}). 
After replacing $(X,\Delta)$ by a quasi-\'etale cover $(X',\Delta')$, our pair $(X',\Delta')$ is still klt with $-(K_{X'}+\Delta')$ being strictly nef (cf.~\cite[Proposition 5.20]{KM98}).
Furthermore, if $X'$ is rationally connected, then so is $X$.
Therefore, we are free to replace $X$ by its quasi-\'etale cover.

By \cite[Theorem 1.1]{MW21}, with $(X,\Delta)$ replaced by a quasi-\'etale cover, the maximal rationally connected (MRC for short) fibration $f:(X,\Delta)\to Y$ of $X$ is a (holomorphic) locally constant fibration with respect to the pair $(X,\Delta)$ such that $Y$ has only klt singularities and the canonical divisor $K_Y\equiv 0$. 
Then by taking a desingularization of $Y$ and considering the base change of $f$, from {\hyperref[t.structureantinef]{Theorem \ref*{t.structureantinef}}} and \cite[Theorem 1.1]{Tak03} we see that there is a $G$-linearized ample line bundle over $F$, and hence all the assumptions in {\hyperref[a.assumption]{Assumption \ref*{a.assumption}}} are satisfied.  

%Besides, since $(X,\Delta)$ is a klt pair,  so is $(F,\Delta_F)$. Therefore, $F$ has rational singularities and hence $q(F)=0$ by considering a desingularization (which is still rationally connected); see \cite[Theorem 5.22]{KM98}. Then it follows from \cite[Proof of Lemma 2.7]{MW21} that  there exists an ample  divisor $A_F$ on $F$ which is $G$-linearized; in particular, all the assumptions in {\hyperref[a.assumption]{Assumption \ref*{a.assumption}}} are satisfied. Moreover, by {\hyperref[cor-lc-periodic]{Corollary  \ref*{cor-lc-periodic}}}, the fundamental group $G=\pi_1(Y)$ has a periodic point on $F$. 

Suppose the contrary that $X$ is not rationally connected.
Then $\dim Y>0$.
By {\hyperref[p.solvableflatsection]{Proposition  \ref*{p.solvableflatsection}}}, up to replacing $X$ by a further \'etale cover, our $f$ admits a section $\sigma:Y\to X$ such that $\sigma^*(K_X+\Delta)\equiv 0$. 
We pick a curve $C$ in $Y$. 
Then $\sigma_\ast C\neq 0$. 
Since $\sigma^\ast(K_X+\Delta)\equiv0$, it follows from the projection formula that
\[
-\sigma^\ast(K_X+\Delta)\cdot C=-(K_X+\Delta)\cdot\sigma_\ast C=0,
\]
which contradicts the strict nefness of $-(K_X+\Delta)$. 
Hence, $X$ is rationally connected and the first part of {\hyperref[mainthm-RC]{Theorem \ref*{mainthm-RC}}} follows. 

Note that a rationally connected klt projective variety has the vanishing irregularity (cf.~\cite[Corollary 4.18]{Deb01} and \cite[Theorem 5.22]{KM98}). Thus, to show the augmented irregularity $q^\circ(X)$ vanishing, we only need to show every quasi-\'etale cover is rationally connected.
However, this follows from the first part of {\hyperref[mainthm-RC]{Theorem \ref*{mainthm-RC}}} and arguments in the beginning of our proof.
\end{proof}

In what follows, we slightly extends our  \hyperref[mainthm-RC]{Theorem \ref*{mainthm-RC}}   
to the following \hyperref[prop-aug-irrg-bc]{Proposition \ref*{prop-aug-irrg-bc}}  on the anti-log canonical divisor being almost strictly nef (cf.~\hyperref[defn-almost-sn]{Definition \ref*{defn-almost-sn}}).

%It is conjectured that if $L_X$ is almost strictly nef on a smooth projective variety $X$, then $K_X+tL_X$ is big for sufficiently large $t\gg 1$ (cf.~\cite[Conjecture 2.2]{CCP08} and \cite[Theorem 26]{Cha20}). 
%When $L_X$ is the ant-log canonical divisor, we now describe the geometric property of such $X$ and give a partial answer to this conjecture  (cf.~\hyperref[cor-3fold-almost-big]{Corollary \ref*{cor-3fold-almost-big}}).

\begin{prop}\label{prop-aug-irrg-bc}
Let $(X,\Delta)$ be a projective klt pair with $-(K_X+\Delta)$ being almost strictly nef.
Then, any quasi-\'etale cover of $X$ is rationally connected.
In particular, the augmented irregularity $q^\circ(X)=0$.
\end{prop}

\begin{proof} 
Let $\pi:(X,\Delta)\to (Z,\Delta_Z)$ be the projective birational morphism such that $K_X+\Delta=\pi^*(K_Z+\Delta_Z)$ and $-(K_Z+\Delta_Z)$ is strictly nef.
Clearly, $(Z,\Delta_Z)$ is also a projective klt pair.
Let $(X',\Delta':=\tau^*\Delta)\xrightarrow{\tau}(X,\Delta)$  be any quasi-\'etale cover. 
Then $(X',\Delta')$ is a projective klt pair (cf.~\cite[Proposition 5.20]{KM98}) and $-(K_{X'}+\Delta')$ is nef.
Taking the Stein factorization of $\pi\circ\tau$, we get the following commutative diagram
\[\xymatrix{
(X',\Delta')\ar[r]^{\pi'}\ar[d]_\tau&(Z',\Delta_{Z'})\ar[d]^{\tau_Z}\\
(X,\Delta)\ar[r]_\pi&(Z,\Delta_Z)
}
\]
where $\tau_Z$ is a finite morphism and $\pi'$ is birational.
By the projection formula, we have $$K_{Z'}+\Delta_{Z'}:=\tau_Z^*(K_Z+\Delta_Z)=\pi'_*(K_{X'}+\Delta').$$
Then $(Z',\Delta_{Z'})$ is also a projective klt pair with $-(K_{Z'}+\Delta_{Z'})$ being strictly nef.
Applying \hyperref[mainthm-RC]{Theorem \ref*{mainthm-RC}} to the pair $(Z',\Delta_{Z'})$, we see that $Z'$ and hence $X'$ are rationally connected, which completes the proof of the proposition . 
\end{proof}

As a  consequence of
{\hyperref[prop-aug-irrg-bc]{Proposition \ref*{prop-aug-irrg-bc}}}
 and \cite[Proposition 7.8]{GKP16b}, we end up the first part of this section with  the following result on the finiteness of the algebraic fundamental group of the smooth locus when  $-(K_X+\Delta)$ is almost strictly nef.

\begin{prop}\label{lem-finite-reg-algpi1}
Let $(X,\Delta)$ be a projective klt pair with $-(K_X+\Delta)$ being almost strictly nef.
Then the algebraic fundamental group $\pi_1^{\textup{alg}}(X_{\textup{reg}})$ of the smooth locus $X_{\textup{reg}}\subseteq X$ is finite.
\end{prop}

\begin{proof}
Denote by $\mathcal{R}$ the set of  normal projective varieties $X$ such that there exists a $\mathbb{Q}$-Weil divisor $\Delta$ such that $(X,\Delta)$ is klt and the anti-log canonical divisor $-(K_X+\Delta)$ is almost strictly nef.
Then it follows from {\hyperref[prop-aug-irrg-bc]{Proposition \ref*{prop-aug-irrg-bc}}} and its proof that any quasi-\'etale Galois cover $Y$ of $X$ still lies in $\mathcal{R}$.
Moreover, since $X$ is rationally connected by {\hyperref[prop-aug-irrg-bc]{Proposition \ref*{prop-aug-irrg-bc}}},
 and the algebraic fundamental group is a profinite completion of the topological fundamental group, we see that $\pi_1^\textup{alg}(X)$ is finite  (cf.~\cite{Tak03}).
In particular, the conditions of \cite[Proposition 7.8]{GKP16b} are verified and thus $\pi_1^{\textup{alg}}(X_{\textup{reg}})$ is finite.
\end{proof}

\vskip 1\baselineskip

In the second part of this section, we aim to show {\hyperref[t.mainthm]{Theorem \ref*{t.mainthm}}} and {\hyperref[c.algebraicallyintegrable]{Corollary \ref*{c.algebraicallyintegrable}}}.
In the view of {\hyperref[t.structureantinef]{Theorem \ref*{t.structureantinef}}}, 
we already know that such $f$ (in {\hyperref[t.mainthm]{Theorem \ref*{t.mainthm}}}) is a locally constant fibration.  
Thus, to prove {\hyperref[t.mainthm]{Theorem \ref*{t.mainthm}}}, it remains to show that the fibres of $f$ are rationally connected and $Y$ is a canonically polarized hyperbolic manifold. 
Indeed, we shall prove that every irreducible subvariety $Z$ of $Y$ is of general type in the sense that every resolution $Z'\rightarrow Z$ is of general type.  

Now, we want to apply {\hyperref[p.solvableflatsection]{Proposition  \ref*{p.solvableflatsection}}} to the situation of {\hyperref[t.mainthm]{Theorem \ref*{t.mainthm}}}.  
First, the following lemma is an application of \cite[Theorem 1.1]{Yam10}, which reveals the relationships between fundamental groups and degeneracy of entire curves.

\begin{lemme}[{\cite[Lemma 9.1]{LOY20}}]
	\label{l.degenerencyentirecurve}
	Let $Y$ be an irreducible projective variety. If there exists a representation $\rho:\pi_1(Y)\rightarrow \GL(r,\CC)$ such that its image $\Image(\rho)$ is not virtually solvable, then every holomorphic map $f:\CC\rightarrow Y$ is degenerate, i.e., $f(\CC)$ is not Zariski dense in $Y$.
\end{lemme}

The following theorem will be used in the proof of {\hyperref[t.mainthm]{Theorem \ref*{t.mainthm}}}. 
Before giving the statement, let us introduce some necessary notion. Let $L$ be a nef $\QQ$-Cartier $\QQ$-Weil divisor on a normal projective variety. We define $\textbf{NT}(L)$ to be the Zariski closure of the union of curves $C$ in $X$ with $L$-degree $0$, i.e., $L\cdot C=0$.

\begin{thm}
	\label{t.generaltype}
	Under {\hyperref[a.assumption]{Assumption \ref*{a.assumption}}}, if the subvariety $T:=\textrm{\bf NT}(-(K_{X/Y}+\Delta))$ does not dominate $Y$ and $-(K_{X/Y}+\Delta)$ is strictly nef, then $Y$ is of general type and there exists a linear representation $\chi:\pi_1(Y)\rightarrow \GL(r,\CC)$ whose image is not virtually solvable.
\end{thm}

\begin{proof}
	Let $\chi:\pi_1(Y) \rightarrow \GL(r,\CC)$ be the linear representation given in  {\hyperref[a.assumption]{Assumption \ref*{a.assumption}}}.
	
	\bigskip
	
  \noindent
	\textit{Step 1. Let $Z\subseteq Y$ be an irreducible subvariety of positive dimension passing through a very general point of $Y$, and let $Z'\rightarrow Z$ be an arbitrary resolution. Then the image of the induced composed map 
		\[
		\chi':\pi_1(Z')\rightarrow \pi_1(Y)\xrightarrow{\chi} \textup{GL}(r,\CC)
		\]
    is not virtually solvable.}

    \bigskip
     
    Assume to the contrary that the image of $\chi'$ is virtually solvable. 
    Since $T$ does not dominate $Y$ and $Z$ is very general, we may assume that $Z$ is not contained in the image of $T$ in $Y$. Let us denote by $f':X'\rightarrow Z'$ the locally constant fibration with respect to $(X',\Delta')$ induced by $\rho'$ and $f$. 
    By {\hyperref[l.basechange]{Lemma \ref*{l.basechange}}}, we have
    \[
    g^*(K_{X/Y}+\Delta) = K_{X'/Z'}+\Delta',
    \]
    where $g$ is the induced morphism $X'\rightarrow X$. On the other hand, by {\hyperref[p.solvableflatsection]{Proposition \ref*{p.solvableflatsection}}}, after replacing $Z'$ by a finite \'etale cover if necessary, there exists a section $\sigma:Z'\rightarrow X'$ such that the pull-back $\sigma^*(K_{X'/Z'}+\Delta')$ is numerically trivial. Since $Z'\rightarrow Z$ is birational and $\dim(Z)>0$, there exists an irreducible curve $C'\subseteq Z'$ such that it is birational onto its image $C$ in $Z$ such that $C$ is not contained in the image of $T$. In particular, the induced morphism $C'\rightarrow \sigma(C')\rightarrow (g\circ \sigma)(C')$ is birational and $(g\circ \sigma)(C')$ is not contained in $T$. Then by projection formula we get
    \[
    0<-(K_{X/Y}+\Delta)\cdot (g\circ \sigma)(C') = -(K_{X'/Z'}+\Delta')\cdot \sigma(C)=0,
    \]
    which is a contradiction.
    
    \bigskip
    
    \noindent
    \textit{Step 2. The base manifold $Y$ is of general type.}
    
    \bigskip
    
    Let $Y'\rightarrow Y$ be the finite \'etale cover of $Y$ given in the discussion before {\hyperref[t.basemfdShafarevichmap]{Theorem  \ref*{t.basemfdShafarevichmap}}} with the Shafarevich map 
    $\text{sh}_{Y'}^K:Y'\dashrightarrow \text{Sh}^K(Y')$ corresponding to the representation $\chi$. 
    Let $G_{Y'}$ be the image of $\pi_1(Y')$ under the composite map $\pi_1(Y')\to\pi_1(Y)\xrightarrow{\chi} \text{GL}(r,\mathbb{C})$ with $\overline{G_{Y'}}$ being its Zariski closure. 
    To show that $Y$ is of general type, it is enough to show that $Y'$ is general. 
    In particular, by {\hyperref[t.basemfdShafarevichmap]{Theorem  \ref*{t.basemfdShafarevichmap}}}, it remains to show that the Shafarevich map $\text{sh}_{Y'}^K$ is birational. 
    Let $X'\rightarrow Y'$ be the locally constant fibration induced by $f$. 
    By {\hyperref[l.basechange]{Lemma \ref*{l.basechange}}}, the relative anti-log canonical divisor $-(K_{X'/Y'}+\Delta')$ is nef and it is clear that the closed subvariety $T':=\textbf{NT}(-(K_{X'/Y'}+\Delta'))$ does not dominate $Y'$. Let $Z$ be a very general fibre of $\text{sh}_{Y'}^K$. Then $Z$ is normal as $X'$ is normal. Let $Z'\rightarrow Z$ be a resolution. Then by {\hyperref[defn-shafarevich]{Definition \ref*{defn-shafarevich}}} the image of the following composed map
    \[
    \pi_1(Z')\rightarrow \pi_1(Z)\rightarrow \pi_1(Y') \rightarrow \pi_1(Y)\xrightarrow{\chi} \overline{G_{Y'}} \rightarrow \overline{G_{Y'}}/\text{Rad}(\overline{G_{Y'}})
    \]
    is finite and hence the image of the following composed map
    \[
    \pi_1(Z')\rightarrow \pi_1(Z)\rightarrow \pi_1(Y') \rightarrow \pi_1(Y)\xrightarrow{\chi} \textup{GL}(r,\mathbb{C})
    \]
    is virtually solvable. Here, $\text{Rad}(\overline{G_{Y'}})$ denotes the solvable radical of $\overline{G_{Y'}}$.
    Hence, by Step 1 above, we obtain $\dim(Z)=0$ and it follows that $\text{sh}_{Y'}^K$ is birational.
\end{proof}

Now we are in the position to prove {\hyperref[t.mainthm]{Theorem \ref*{t.mainthm}}}.

\begin{proof}[Proof of {\hyperref[t.mainthm]{Theorem \ref*{t.mainthm}}}]
	By {\hyperref[t.structureantinef]{Theorem \ref*{t.structureantinef}}}, it remains to show that $Y$ is a hyperbolic manifold with ample canonical divisor and the fibres of $f$ is rationally connected. 
	Firstly note that $(F,\Delta_F)$ is a  projective klt pair with $-(K_F+\Delta_F)$ being strictly nef, where $F$ is a general fibre of $f$. 
	Hence, the variety $F$ is rationally connected by {\hyperref[mainthm-RC]{Theorem \ref*{mainthm-RC}}}. Moreover, by {\hyperref[t.generaltype]{Theorem  \ref*{t.generaltype}}}, it is known that the canonical divisor $K_Y$ is big. 
	Thus to show that $K_Y$ is ample, it is enough to show that $Y$ does not contain any rational curves, 
	which shall be a direct consequence of the hyperbolicity of $Y$. 
	Indeed, if $K_Y$ is not nef, then it follows from the cone theorem (cf.~\cite[Theorem 3.7]{KM98}) that $Y$ contains a rational curve; if $K_Y$ is nef (and big) but not ample, then it follows from the base-point-free theorem (cf.~\cite[Theorem 3.3]{KM98}) that $Y$ admits a birational holomorphic Kodaira fibration onto its canonical model $Y^{\text{can}}$ (with only canonical singularities), in which case, the exceptional locus of $Y\to Y^{\text{can}}$ is covered by rational curves (cf.~\cite[Proof of Proposition 1.3]{KM98}). 
Thus, it is enough to show that $Y$ is hyperbolic.
	
	Let $h:\mathbb{C}\rightarrow Y$ be a holomorphic map and denote by $Z$ the Zariski closure of its image. Assume to the contrary that $\dim(Z)>0$. Then $f(\CC)$ is not degenerate in $Z$. Let $Z'$ be a resolution of $Z$ and let $X'\rightarrow Z'$ be the locally constant fibration induced by $f$. Then applying {\hyperref[t.generaltype]{Theorem  \ref*{t.generaltype}}}, we see that there exists a linear representation of $\pi_1(Z')$ whose image is not virtually solvable. 
	On the other hand, note that the holomorphic map $h$ can be lifted to a holomorphic map $h':\CC\rightarrow Z'$. In particular, the image $h'(\CC)$ is dense in $Z'$, which contradicts {\hyperref[l.degenerencyentirecurve]{Lemma  \ref*{l.degenerencyentirecurve}}}. Hence, we have $\dim(Z)=0$; that is, $h$ is constant and consequently $Y$ is hyperbolic.
\end{proof}

\begin{rmq}
	From the proof, one can easily derive the following result concerning   the geometry of $Y$: for an irreducible closed subvariety $Z\subseteq Y$ of positive dimension and with a resolution $Z'\rightarrow Z$, then $Z'$ is of general type and $\pi_1(Z')$ admits a linear representation whose image is not virtually solvable.
\end{rmq}

Let $X$ be a projective manifold. A \emph{foliation} on $X$ is a coherent subsheaf $\scrF$ of $T_X$ such that $\scrF$ is closed under the Lie bracket and the quotient $T_X/\scrF$ is torision free. The canonical divisor $K_{\scrF}$ of $\scrF$ is a Weil divisor on $X$ such that $\scrO_X(-K_{\scrF})\cong \det\scrF$. We call that $\scrF$ is \emph{regular} if the quotient $T_{X}/\scrF$ is locally free and $\scrF$ is \emph{algebraically integrable} if there exists a rational map $f:X\dashrightarrow Y$ such that $\scrF$ is induced by $f$.

\begin{proof}[Proof of {\hyperref[c.algebraicallyintegrable]{Corollary \ref*{c.algebraicallyintegrable}}}]
	Assume that either $\scrF$ is regular, or $\scrF$ has a compact leaf. By \cite[Theorem 1.1]{Ou21}, there exists a locally trivial fibration $f:X\rightarrow Y$ with rationally connected fibres such that there exists a foliation $\scrG$ on $Y$ with $K_{\scrG}\equiv 0$ and $\scrF=f^{-1}(\scrG)$. In particular, we have $K_{\scrF}\equiv K_{X/Y}$ and therefore $-K_{X/Y}$ is strictly nef. 
	
	For the statement (1), by {\hyperref[t.mainthm]{Theorem \ref*{t.mainthm}}}, the base manifold $Y$ is a canonically polarized projective manifold. In particular, the tangent bundle $T_Y$ is semi-stable with respect to $K_Y$ by the theorem of Aubin-Yau (\cite{Aub78,Yau78}). Observing that the slope of $T_Y$ with respect to $K_Y$ is strictly negative, we see that $\scrG$ is a fortiori the foliation in points and hence $\scrF$ is exactly the foliation induced by the fibration $f:X\rightarrow Y$.
	
	Next we assume that the fundamental group $\pi_1(X)$ is virtually solvable. Then the fundamental group $\pi_1(Y)=\pi_1(X)$ of $Y$ is virtually solvable since the fibres of $f$ are rationally connected. In particular, the image of every linear representation of $\pi_1(Y)$ is again virtually solvable, which contradicts {\hyperref[t.generaltype]{Theorem  \ref*{t.generaltype}}}.
\end{proof}

\begin{exa}
	We collect some examples to show the sharpness of {\hyperref[c.algebraicallyintegrable]{Corollary \ref*{c.algebraicallyintegrable}}}.
	\begin{enumerate}
		\item[(1)] Let $\scrF\subseteq T_{\PP^n}$ be the foliation of rank $r$ defined by a linear projection $\PP^n\dashrightarrow \PP^{n-r}$. Then  $\det\scrF\cong \scrO_{\PP^n}(r)$ is ample. This shows that the regularity of $\scrF$ in {\hyperref[c.algebraicallyintegrable]{Corollary \ref*{c.algebraicallyintegrable}}}  cannot be dropped.
		
		\item[(2)] Let $X$ be an abelian variety and let $\scrF$ be a general linear foliation on $X$, that is, the translation of a general linear space. Then $-K_{\scrF}$ is numerically trivial and $\scrF$ is not algebraically integrable. 
		This shows that we cannot replace strict nefness by nefness in {\hyperref[c.algebraicallyintegrable]{Corollary \ref*{c.algebraicallyintegrable} (1)}}.
		
		\item[(3)] Let $\scrF$ be the foliation induced by the fibration $X=\PP E\to C$, which is Mumford's example discussed after {\hyperref[t.AD]{Theorem  \ref*{t.AD}}}. Then $-K_{\scrF}$ is strictly nef. This shows that the assumption on the fundamental group of $X$ cannot be removed in {\hyperref[c.algebraicallyintegrable]{Corollary \ref*{c.algebraicallyintegrable} (2)}}.
	\end{enumerate}
\end{exa}

\vskip 2\baselineskip

\section{Ampleness of \texorpdfstring{$-(K_X+\Delta)$}{text}, Proof of \texorpdfstring{{\hyperref[mainthm-ample-3klt]{Theorem \ref*{mainthm-ample-3klt}}}}{text}}\label{section-pf-thmc}

In this section, we are devoted to the proof of  \texorpdfstring{{\hyperref[mainthm-ample-3klt]{Theorem \ref*{mainthm-ample-3klt}}}}{text}.
Let $(X,\Delta)$ be a projective klt threefold pair such that the anti-log canonical divisor $-(K_X+\Delta)$ is strictly nef.
%The whole section is devoted to the proof of {\hyperref[mainthm-ample-3klt]{Theorem \ref*{mainthm-ample-3klt}}}. 
 
%To prove \texorpdfstring{{\hyperref[mainthm-ample-3klt]{Theorem \ref*{mainthm-ample-3klt}}}}{text}, we need to do some preparations by recalling and also generalizing some previous results to the singular (pair) setting. 
%Second, we apply such results to prove \texorpdfstring{{\hyperref[mainthm-ample-3klt]{Theorem \ref*{mainthm-ample-3klt}}}}{text}. 
\subsection{\texorpdfstring{$\mathbb{Q}$}{text}-effectivity of \texorpdfstring{$-(K_X+\Delta)$}{text}}
To show the ampleness of $-(K_X+\Delta)$, we aim to show that $-(K_X+\Delta)$ is \textit{num-effective}, i.e., numerically equivalent to an effective divisor, which is our main result  {\hyperref[prop-klt-eff-ample]{Proposition \ref*{prop-klt-eff-ample}}} in this subsection.
Before that, we recall and generalize some previous results for the convenience of our proofs.

First we give the following generalization of \cite[Lemma 1.1]{Ser95}.

\begin{lemme}[{cf.~\cite[Lemma 1.1]{Ser95}}]\label{lem_strict_nef_k}
Let $(X,\Delta)$ be a projective klt pair, and $L_X$  a strictly nef $\mathbb{Q}$-divisor on $X$.
Then $(K_X+\Delta)+tL_X$ is  strictly nef for every $t>2m\dim X$ where $m$ is the Cartier index of $L_X$.
\end{lemme}

\begin{proof}
By the cone theorem (cf.~\cite[Theorem 3.7, p.~76]{KM98}), every curve $C$ in $X$ is numerically equivalent to a linear combination 
$M+\sum a_iC_i$ (a finite sum), where  $M$ is  pseudo-effective such that $M\cdot (K_X+\Delta)\geqslant 0$ and the $C_i$'s are (integral) rational curves satisfying $0<C_i\cdot (-K_X-\Delta)\leqslant2\dim X$	 with $a_i>0$.
Since $mL_X$ is Cartier and strictly nef,  we have $L_X\cdot M\geqslant 0$, and $L_X\cdot C_i\geqslant \frac{1}{m}$ for all $i$.
Therefore, for each $C_i$, 
\begin{align}\tag{\dag}\label{eq_lem_strict}
((K_X+\Delta)+tL_X)\cdot C_i\geqslant tL_X\cdot C_i-2\dim X>0.
\end{align}
If $(K_X+\Delta)\cdot C\geqslant 0$, then $((K_X+\Delta)+tL_X)\cdot C>0$ since $L_X\cdot C>0$.
If $(K_X+\Delta)\cdot C<0$, then the decomposition of $C$ contains some $C_i$; hence our lemma follows from (\ref{eq_lem_strict}).
\end{proof}

We remark here that, when $\dim X=2$, the conclusion of  {\hyperref[lem_strict_nef_k]{Lemma  \ref*{lem_strict_nef_k}}} holds for a more optimal lower bound $t>3$ (cf.~\cite[Proposition 3.8]{Fuj12}).

\begin{lemme}\label{lem-big-ample}
Let $(X,\Delta)$ be a projective klt pair, 
and $L_X$  a strictly nef $\mathbb{Q}$-divisor on $X$.	
Suppose  $a(K_X+\Delta)+bL_X$ is big for some $a,b\geqslant 0$.
Then $(K_X+\Delta)+tL_X$ is ample for sufficiently large $t$. 
\end{lemme}

\begin{proof}
If $a=0$, then $L_X$ is big. %, hence it follows from 
If $a\neq 0$, then $(K_X+\Delta)+\frac{b}{a}L_X$ is big.
In both cases, it follows from  {\hyperref[lem_strict_nef_k]{Lemma  \ref*{lem_strict_nef_k}}} that $(K_X+\Delta)+tL_X$ is strictly nef and big for $t\gg 1$, noting that nef divisors are always pseudo-effective.
Then $2((K_X+\Delta)+tL_X)-(K_X+\Delta)$ is also nef and big.
By the base-point-free theorem (cf.~\cite[Theorem 3.3, p.~75]{KM98}), some multiple of $(K_X+\Delta)+tL_X$ defines a morphism $X\to\mathbb{P}^N$, which is finite by the projection formula. Therefore, $(K_X+\Delta)+tL_X$ is ample, and our lemma is proved.
\end{proof}

We denote by $\overline{\textup{NE}}(X)$ (resp. $\overline{\textup{ME}}(X)$) the \textit{Mori cone} (resp. \textit{movable cone}) of a projective variety $X$ (cf.~\cite{BDPP13}).
The following lemma characterizes the situation when $K+tL$ is not big. 
\begin{lemme}\label{lem-not-big-some}
Let $(X,\Delta)$ be a projective klt  pair of dimension $n$, 
and $L_X$  a strictly nef $\mathbb{Q}$-divisor on $X$.	
Then $(K_X+\Delta)+uL_X$ is not big for some rational number $u>2mn$ with $m$  the Cartier index of $L_X$, if and only if
$(K_X+\Delta)^i\cdot L_X^{n-i}=0$ for any $0\leqslant i\leqslant n$. 	
Moreover, if one of the  equivalent conditions holds, there exists a  class $0\neq \alpha\in\overline{\textup{ME}}(X)$ such that 
$(K_X+\Delta)\cdot\alpha=L_X\cdot\alpha=0$. 
\end{lemme}

\begin{proof}
Suppose first that $(K_X+\Delta)+uL_X$ is not big for some rational number $u>2mn$. 
Set $u':=u-2mn>0$, $D_1:=(K_X+\Delta)+(u'/2+2mn)L_X$ and $D_2:=(u'/2)L_X$.	
Then $(K_X+\Delta+uL_X)^n=(D_1+D_2)^n=0$.
Since both $D_1$ and $D_2$ are nef (cf.~{\hyperref[lem_strict_nef_k]{Lemma  \ref*{lem_strict_nef_k}}}), the vanishing $D_1^i\cdot D_2^{n-i}=0$ for every $0\leqslant i\leqslant n$  yields $(K_X+\Delta)^i\cdot L_X^{n-i}=0$ for any $0\leqslant i\leqslant n$. 
Conversely, if $(K_X+\Delta)^i\cdot L_X^{n-i}=0$ for any $0\leqslant i\leqslant n$, then $((K_X+\Delta)+uL_X)^n=0$ for any $u$; in this case, $(K_X+\Delta)+uL_X$ being not big for any $u>2mn$ follows from \cite[Proposition 2.61, p.~68]{KM98} and {\hyperref[lem_strict_nef_k]{Lemma  \ref*{lem_strict_nef_k}}}. 
Therefore, the first part of our theorem is proved.

Assume that the nef divisor $K_X+\Delta+uL_X$ is not big for some $u>2mn$. 
Then, there exists a class $\alpha\in\overline{\textup{ME}}(X)$ such that $(K_X+\Delta+uL_X)\cdot\alpha=0$ (cf.~\cite[Theorem 2.2 and the remarks therein]{BDPP13}). 
%Indeed, we can choose $\alpha$ to be a rational class.
%Neither $K_X+\Delta+uL_X$ nor $-(K_X+\Delta+uL_X)$ being strictly positive  on $\overline{\textup{ME}}(X)$, there are rational points $P,Q\in\overline{\textup{ME}}(X)$ such that $(K_X+\Delta+uL_X)\cdot P\geqslant 0$ and $(K_X+\Delta+uL_X)\cdot Q\leqslant0$, noting that rational points are dense in $\overline{\textup{ME}}(X)$.
%If one of the inequalities is an equality, then we are done.
%Otherwise, we take $\mu:=-\frac{((K_X+\Delta+uL_X)\cdot P)}{((K_X+\Delta+uL_X)\cdot Q)}>0$ (a rational number) and $\alpha:=P+\mu Q\in\overline{\textup{ME}}(X)$, which is a rational class in $\overline{\textup{ME}}(X)$ satisfying $(K_X+\Delta+uL_X)\cdot\alpha=0$.
Suppose that $L_X\cdot\alpha\neq 0$.  
Then $L_X\cdot\alpha>0$ and thus $(K_X+\Delta)\cdot\alpha<0$.
By the cone theorem (cf.~\cite[Theorem 3.7, p.~76]{KM98}), we have 
$\alpha=M+\sum a_iC_i$ 
where $M$ is a class lying in $\overline{\textup{NE}}(X)_{(K_X+\Delta)\geqslant 0}$, $a_i>0$, and the $C_i$ are extremal rational curves with $0<-(K_X+\Delta)\cdot C_i\leqslant 2n$.
%Since $(K_X+\Delta)\cdot \alpha<0$, there exists a non-zero $i$, i.e., $k>0$.
Now, for each $i$, the intersection $(K_X+\Delta+uL_X)\cdot C_i>0$, a contradiction to the choice of $\alpha$.
So we have $(K_X+\Delta)\cdot \alpha=L_X\cdot\alpha=0$.
\end{proof}

%Now we give an alternative proof of {\hyperref[main_thm_surface]{Proposition  \ref*{main_thm_surface}}} (cf.~\cite[Corollary 1.8]{HL20}) by applying a recent result on \textit{almost strictly nef} divisors on surfaces. 
%; cf. Proposition \ref{prop_Q-Goren_surface} for an extension to the non-normal case. 
%The proof depends on a recent progress on the bigness of   \textit{almost strictly nef} divisors on smooth projective varieties (cf.~\cite[Theorem 20]{Cha20} and \cite[Proposition 2.3]{CCP08}). %

%\begin{proof}[\textup{\textbf{Proof of {\hyperref[main_thm_surface]{Proposition \ref*{main_thm_surface}}}}}]
%Taking a minimal resolution $\pi:\widetilde{X}\to X$, we see that $L_{\widetilde{X}}:=\pi^*L_X$ is almost strictly nef. 
%By \cite[Theorem 26]{Cha20},  $K_{\widetilde{X}}+tL_{\widetilde{X}}$ is big for all $t>3$.
%Then its push-down $K_X+tL_X$	 is big (as a Weil divisor) for  all $t>3$ (cf.~\cite[Lemma 4.10 (i\!i\!i) and its proof]{FKL16}), i.e., $K_X+tL_X$ lies in the interior part of the closure of the cone generated by effective Weil divisors on $X$ for all $t>3$.  
%So $K_X+\Delta+tL_X$ is big (as a Cartier divisor) for all $t>3$.
%Since $L_X$ is a strictly nef Cartier divisor, by  {\hyperref[lem-big-ample]{Lemma  \ref*{lem-big-ample}}} and the remark after it, we see that for $t>3$,  $K_X+\Delta+tL_X$ is strictly nef and $K_X+\Delta+2tL_X=2(K_X+\Delta+tL_X)-(K_X+\Delta)$ is  nef and big.  
%Applying the base-point-free theorem for $K_X+\Delta+tL_X$ (cf.~\cite[Theorem 3.3, p. 75]{KM98}), %(cf.~{\hyperref[lem_strict_nef_k]{Lemma  \ref*{lem_strict_nef_k}}}) for $t\gg 1$,
% our theorem follows from the strict nefness of $K_X+\Delta+tL_X$ for $t>3$ (cf. {\hyperref[lem-big-ample]{Lemma  \ref*{lem-big-ample}}} and the remark after it).   
%The second part  follows from \cite{Zhang06}.
%\end{proof}

In the following, we slightly generalize \cite[Theorem 2.3]{Ser95} and \cite[Corollary 1.8]{HL20}  %{\hyperref[main_thm_surface]{Proposition \ref*{main_thm_surface}}}   
to the following (not necessarily normal)  $\mathbb{Q}$-Gorenstein surface case (cf.~\cite[Conjecture 1.3]{CCP08}). 
This is also mentioned at the end of \cite[Section 2]{Ser95}.

For a $\mathbb{Q}$-factorial normal projective variety $X$ and a prime divisor $S\subseteq X$,  the \textit{canonical divisor} $K_S\in\textup{Pic}(S)\otimes\mathbb{Q}$ is defined by  
$K_S:=\frac{1}{m}(mK_X+mS)|_S$, where $m\in\mathbb{N}$ is the smallest positive integer such that both $mK_X$ and $mS$ are Cartier divisors on $X$.

\begin{prop}\label{prop_Q-Goren_surface}
Let $(X,\Delta)$ be a $\mathbb{Q}$-factorial klt threefold pair,  $S$ a prime divisor on $X$, and  $L_S$ a strictly nef  divisor (resp. almost strictly nef) on $S$.
Then $K_S+tL_S$ is ample (resp. big) for sufficiently large $t$.
\end{prop}
\begin{proof}
Since $X$ is $\mathbb{Q}$-factorial, there exists some positive integer $m$ such that $m(K_X+S)$ is Cartier. 
Let $j:T\to S$ be the normalization, and $f:R\to T$  a minimal resolution.

First, we show the case when $L_S$ is strictly nef. 
In view of \cite[Proof of Theorem 2.3]{Ser95}, we only need to reprove the following injection	in \cite[Claim 4 of Proof of Theorem 2.3]{Ser95} for sufficiently large $r\gg 1$:
\begin{equation}\label{eq-inclusion}
f_*\scrO_R(rmK_R)\hookrightarrow j^*\scrO_S(rmK_S).
\end{equation} %, and $\pi$ is the composition. 
Since $(X,\Delta)$ is a dlt pair,  $X$ is Cohen-Macaulay (cf.~\cite[Theorems 5.10 and 5.22]{KM98}).
By \cite[Lemma 5-1-9]{KMM87}, there is a natural injective homomorphism
$$\omega_T^{[m]}:=(\omega_T^{\otimes m})^{\vee\vee}=\scrO_T(mK_T)\hookrightarrow j^*\scrO_X(m(K_X+S))=j^*\scrO_S(mK_S).$$
On the other hand, it follows from \cite[Theorem (2.1)]{Sak84} that
%taking the double dual, injection $\scrO_{\widetilde{E}}(K_{\widetilde{E}})\hookrightarrow \scrO(\sigma^*K_{E_N})=\sigma^*\scrO(K_{E_N})^{\vee\vee}\hookrightarrow \sigma^*n_E^*\scrO(K_E)$??? $\sigma^*n_E^*\scrO(K_E)$ is locally free and hence reflexive, reflexive hull? 
$$f_*\scrO_R(f^*(mK_T))\cong\scrO_T(mK_T).$$
Here, our $mK_T$  is only a Weil divisor. 
Replacing $m$ by a multiple, we write $f^*(mK_T)=mK_R+\Gamma$ with $\Gamma$ being an effective (integral)  divisor (cf.~\cite[(4.1)]{Sak84}).
Then 
$$f_*\scrO_R(mK_R)\subseteq f_*(\scrO_R(mK_R)\otimes\scrO_R(\Gamma))=f_*\scrO_R(f^*(mK_T))\cong\scrO_T(mK_T)\hookrightarrow j^*\scrO_S(mK_S).$$
%note that there is a non-zero trace map $f_*\omega_R\to\omega_T$% (cf.~\cite[Proposition 5.77]{KM98}), which induces a non-zero map 
%$\delta_r:f_*(\omega_R^{\otimes r})\to\scrO_T(rK_T)$ for each $r$.
%Since $f_*(\omega_R^{\otimes r})$ is a torsion free rank one sheaf,  $\delta_r$ is  injective.
Hence, we get the inclusion {\hyperref[eq-inclusion]{(\ref*{eq-inclusion})}} as desired and  our proposition for the case $L_S$ being strictly nef is proved. 

Second, if $L_S$ is almost strictly nef, then so is $f^*j^*L_S$.
By \cite[Theorem 26]{Cha20}, our $K_T+tf^*j^*L_S$ is big for sufficiently large $t\gg 1$.
Applying the inclusion {\hyperref[eq-inclusion]{(\ref*{eq-inclusion})}},  we see that,  for $t\gg 1$,  $K_{S}+tL_S$ is the sum of an effective divisor and the pushforward of some big divisor along a generically finite morphism; hence $K_{F_i}+tL_Y|_{F_i}$ is also big, and our proposition for the case $L_S$ being almost strictly nef is proved.
\end{proof}

In what follows, we  show the main result of this subsection, which  is a key to the proof of  {\hyperref[mainthm-ample-3klt]{Theorem \ref*{mainthm-ample-3klt}}}. 
%However, we are still not able to deal with the general  case when  $aL_X+b(K_X+\Delta)$ is num-effective (when $X$ is non-$\mathbb{Q}$-factorial) due to a gap of the adjunction.

\begin{prop}\label{prop-klt-eff-ample}
Let $(X,\Delta)$ be a projective klt threefold pair.
Let $L_X$ be a strictly nef $\mathbb{Q}$-divisor on $X$.
Suppose that $L_X$ is numerically equivalent to a non-zero effective divisor.
Then $K_X+\Delta+tL_X$ is ample for sufficiently large $t$.
\end{prop}

Before proving {\hyperref[prop-klt-eff-ample]{Proposition \ref*{prop-klt-eff-ample}}}, we  give the lemma below, which can be easily deduced from Zariski decomposition on the surface (or Kodaira's lemma) and the Hodge index theorem. 
\begin{lemme}\label{lem-hodge-nef-big}
Let $L$ be a nef $\mathbb{Q}$-Cartier divisor on a smooth projective surface $S$.
Suppose that $L\cdot B=0$ for some big $\mathbb{Q}$-Cartier divisor $B$ on $S$.
Suppose further that $L^2=0$.
Then $L\equiv 0$.
\end{lemme}

%With all the preparations settled, we will start the proof of   {\hyperref[prop-klt-eff-ample]{Proposition \ref*{prop-klt-eff-ample}}}.

\begin{proof}[Proof of {\hyperref[prop-klt-eff-ample]{Proposition \ref*{prop-klt-eff-ample}}}]
Suppose the contrary that $K_X+\Delta+tL_X$ is not ample for one (and hence any) $t\gg 1$.
By {\hyperref[lem-big-ample]{Lemma \ref*{lem-big-ample}}} and {\hyperref[lem-not-big-some]{Lemma \ref*{lem-not-big-some}}} we have 
\begin{equation}\label{eq_int-nb=0}
(K_X+\Delta)^3=(K_X+\Delta)\cdot L_X^2=(K_X+\Delta)\cdot L_X=L_X^3=0.
\end{equation}
Let us take a $\mathbb{Q}$-factorialization $\pi:(Y,\Gamma)\to (X,\Delta)$, which is a small birational morphism (cf.~\cite[Corollary 1.37, pp.~29-30]{Kollar13}) and let $L_Y:=\pi^*L_X$ which is  nef but not big. 
Then $L_Y^3=0$.

Since $L_X$ is numerically equivalent to a non-zero effective divisor, our $L_Y$ is also numerically equivalent to an effective divisor $\sum n_iF_i$ such that $n_i>0$ for each $i$ and none of $F_i$ is $\pi$-exceptional by the choice of $\pi$.
So it is clear that $L_Y|_{F_i}$ (as the pullback of $L_X|_{\pi(F_i)}$ via $\pi|_{F_i}$) is almost strictly nef (cf.~\hyperref[defn-almost-sn]{Definition \ref*{defn-almost-sn}}). 
By {\hyperref[prop_Q-Goren_surface]{Proposition \ref*{prop_Q-Goren_surface}}}, we see that   $K_{F_i}+tL_Y|_{F_i}$ is also big for $t\gg 1$. 
Since $L_Y^3=0$, we have $(L_Y|_{F_i})^2=L_Y^2\cdot F_i=0$ for every $i$ due to the nefness of $L_Y$. Then by {\hyperref[lem-hodge-nef-big]{Lemma \ref*{lem-hodge-nef-big}}}, we have
$$0<L_Y|_{F_i}\cdot (K_{F_i}+tL_Y|_{F_i})=L_Y\cdot (K_Y+F_i)\cdot F_i.$$
Notice that $L_Y\cdot F_i^2=0$ for every $i$. In fact, otherwise if $L_Y\cdot F_{i_0}^2>0$ for some $i_0$, then 
\[
0=F_{i_0}\cdot L_Y^2=F_{i_0}\cdot \sum n_iF_i\cdot L_Y\geqslant n_{i_0}L_Y\cdot F_{i_0}^2>0,
\]
which is absurd.
Consequently, we have $L_Y\cdot K_Y\cdot F_i>0$ for every $i$, 
and as a result,
$$(K_Y+tL_Y)\cdot L_Y^2=K_Y\cdot L_Y^2=K_Y\cdot L_Y\cdot \sum n_iF_i>0.$$
Since $L_Y$ is nef, the above inequality further implies that 
\[
(K_X+\Delta+tL_X)\cdot L_X^2=(K_Y+\Gamma+tL_Y)\cdot L_Y^2>0,
\]
by the projection formula. This nevertheless contradicts the calculation $L_X^3=(K_X+\Delta)\cdot L_X^2=0$, and the proposition is thus proved.
\end{proof}

\subsection{Proof of \texorpdfstring{{\hyperref[mainthm-ample-3klt]{Theorem \ref*{mainthm-ample-3klt}}}}{text}}
In the beginning, let us recall
the following conjecture on the numerical nonvanishing for generalized polarized pairs, which is closely related to our {\hyperref[main-conj-singular-ample]{Question \ref*{main-conj-singular-ample}}} (cf.~{\hyperref[prop-klt-eff-ample]{Proposition \ref*{prop-klt-eff-ample}}}). 
Recall that the notion of \textit{generalized polarized pairs} was first introduced in \cite{BZ16} to deal with the effectivity of Iitaka fibrations.  %(cf.~\cite{Bir19}). 
%We refer readers to \cite{Kaw98},  and \cite{HL17} for more information and applications on the generalized polarized pairs. 
\begin{conj}[{\cite[Conjecture 1.2]{HL20}}]
Let $(X,B+\mathbf{M})$ be a projective generalized log canonical pair.
Suppose that 
\begin{enumerate}
\item[(i)] $K_X+B+\mathbf{M}_X$ is pseudo-effective; and
\item[(ii)]	$\mathbf{M}=\sum_j\mu_j\mathbf{M}_j$ where $\mu_j\in\mathbb{R}_{>0}$ and $\mathbf{M}_j$ are nef Cartier $b$-divisors.
\end{enumerate}
Then there exists an effective $\mathbb{R}$-Cartier $\mathbb{R}$-divisor $D$ such that $K_X+B+\mathbf{M}_X\equiv D$.
\end{conj}

We shall apply the following theorem to prove our \hyperref[mainthm-ample-3klt]{Theorem \ref*{mainthm-ample-3klt}}. 
Indeed, Han and Liu showed this theorem in a more general setting (\cite[Theorem 4.5]{HL20}) by proving the existence of a good minimal model   after  passing $X$ to a higher model if necessary. 
We refer readers to  \cite[Section 4]{HL20} for more technical details involved. 
\begin{thm}[{cf.~\cite[Theorems 1.5 and 1.11]{HL20}}]\label{thm-hl20-numeff}
Let $(X,\Delta+\mathbf{M}_X)$ be a $3$-dimensional projective generalized klt pair such that $\mathbf{M}_X$ is an $\mathbb{R}_{>0}$-linear combination of nef Cartier $b$-divisors.
Suppose that $K_X+\Delta+\mathbf{M}_X$ is pseudo-effective and $K_X+\mathbf{M}_X$ is not pseudo-effective.
Then $K_X+\Delta+\mathbf{M}_X$ is numerically equivalent to an effective $\mathbb{R}$-Cartier $\mathbb{R}$-divisor.
\end{thm}

\noindent
With all the preparations settled down, 
we begin to prove \hyperref[mainthm-ample-3klt]{Theorem \ref*{mainthm-ample-3klt}}.
First, it follows from \cite[Corollary 1.4.3]{BCHM10} that
we can take a $\mathbb{Q}$-factorial terminalization $\tau:(Y,\Gamma)\to (X,\Delta)$ of $(X,\Delta)$ such that $(Y,\Gamma)$ is a $\mathbb{Q}$-factorial terminal pair.

\begin{thm}\label{thm-case-gammaneq0}
If $\Gamma\neq 0$, then $-(K_X+\Delta)$ is ample.
\end{thm}
\begin{proof}
Let us define $M_t:=-t(K_Y+\Gamma)$ with $t>1$.
Then the sum $K_Y+\Gamma+M_t=(1-t)(K_Y+\Gamma)$ is nef and hence pseudo-effective, and $K_Y+M_t=(1-t)(K_Y+\Gamma)-\Gamma$.
Clearly, one can choose suitable $t\rightarrow 1^+$ such that $K_Y+M_t$ is not pseudo-effective.
Therefore, applying {\hyperref[thm-hl20-numeff]{Theorem  \ref*{thm-hl20-numeff}}}, we see that $-(K_Y+\Gamma)$ is num-effective.
By the projection formula and the strict nefness of $-(K_X+\Delta)$, we see that $-(K_X+\Delta)$ is also num-effective. 
So our theorem follows from {\hyperref[prop-klt-eff-ample]{Proposition  \ref*{prop-klt-eff-ample}}} (with $L_X$ replaced by $-(K_X+\Delta)$ therein).
\end{proof}

\begin{rmq}\label{rmq-case-gamma=0}
If $\Gamma=0$, then our boundary divisor $\Delta=\tau_*\Gamma=0$.
In this case, $\tau$ is  crepant and thus our $X$ has at worst canonical singularities (with strictly nef anti-canonical divisor $-K_X$).
%Indeed, if $E$ is an exceptional divisor over $Y$, then the discrepancy $a(E,Y)=a(E,X)>0$.
%If $E$ is $\tau$-exceptional, then $a(E,X)=0$.
%Therefore, $X$ is a projective  threefold with at worst canonical singularities such that $-K_X$ is strictly nef. 
Then  our \texorpdfstring{{\hyperref[mainthm-ample-3klt]{Theorem \ref*{mainthm-ample-3klt}}}}{text} in this case follows from \cite[Main Theorem]{Ueh00}. 
Since in this case, our \texorpdfstring{{\hyperref[mainthm-RC]{Theorem \ref*{mainthm-RC}}}}{text} covers \cite[Remark 3.8]{Ueh00}.
For the convenience of readers, we shall give a simplified proof of \cite{Ueh00} here.
\end{rmq}

\textbf{In the following, we may assume that $\Gamma=0$.} 
Suppose the contrary that $-K_X$ is not ample.
Then we see have
$$\dimcoh^0(Y,-rK_Y)=\dimcoh^0(X,-rK_X)=0$$ 
where $r$ is the Cartier index of $K_X$ (cf.~{\hyperref[prop-klt-eff-ample]{Proposition \ref*{prop-klt-eff-ample}}}). 
Besides, since $-K_X$ is not ample, it is not big (cf.~{\hyperref[lem-big-ample]{Lemma \ref*{lem-big-ample}}}). 
Hence, $-K_Y$ is nef but not big, which implies that $K_Y^3=0$ (cf.~\cite[Proposition 2.61]{KM98}).  
Moreover, since $Y$ is a $\mathbb{Q}$-factorial terminal threefold such that $-K_Y$ is nef but not big, it follows from \cite[Corollary 6.2]{KMM94} that $-K_Y\cdot \hat{c}_2(Y)\geqslant 0$ where $\hat{c}_2(Y)$ is the second Chern class of the sheaf $\hat{\Omega}_Y^1$.
We note that, when $X$ is smooth in codimension two (which is the case when $X$ has only terminal singularities), $\hat{c}_2(X)$ coincides with the birational section Chern class $c_2(X)$ (cf.~e.g.~\cite[Proof of Claim 4.11]{GKP16b}).

Let us recall the following symbols.  
For a nef $\mathbb{R}$-Cartier $\mathbb{R}$-divisor $L$ on a normal projective variety $Y$,  the 
\textit{numerical dimension} $\nu(Y,L)$ of $L$ is defined to be the maximum $t$ such that $L^t\not\equiv0$.
We denote by $\kappa(Y,L)$ the \textit{Iitaka dimension} of $L$. 
Clearly, 
$\kappa(Y,L)\leqslant\nu(Y,L)\leqslant\dim Y$.

\begin{lemme}\label{lem-gamma0-h2>0}
If $\Gamma=0$ and $-K_Y$ is not big, then $\dimcoh^2(Y,-rK_Y)>0$.
\end{lemme}
\begin{proof}
Since $Y$ is rationally connected by {\hyperref[mainthm-RC]{Theorem \ref*{mainthm-RC}}}, we see that the Euler characteristic $\chi(\scrO_Y)>0$ (cf.~\cite[Corollary 4.18]{Deb01} and \cite[Theorem 5.22]{KM98}).
Applying the Riemann-Roch formula for $-rK_Y$, we get
$$\chi(-rK_Y)=\frac{1}{12}(-rK_Y)\hat{c}_2(Y)+\chi(\scrO_Y)>0,$$
by noting that $K_Y^3=0$ and $\frac{1}{12}(-rK_Y)\hat{c}_2(Y)\geqslant 0$ (cf.~\cite[Corollary 6.2]{KMM94}).
This together with $\dimcoh^0(Y,-rK_Y)=0$ (by our assumption that $-(K_X+\Delta)$ is not ample) implies our lemma.	
\end{proof}

\begin{lemme}\label{lem-gamma0-h1=h2}
If $\Gamma=0$ and $-K_Y$ is not big, then for a sufficiently ample general hyperplane section $S$ on $Y$, we have $\Coh^1(S,(-rK_Y+S)|_S)\simeq\Coh^2(Y,-rK_Y)$.
\end{lemme}
\begin{proof}
We consider the short exact sequence $$0\to\scrO_Y(-rK_Y)\to\scrO_Y(-rK_Y+S)\to\scrO_S((-rK_Y+S)|_S)\to 0$$ 
and the induced long exact sequence
$$\cdots\to \Coh^1(Y,-rK_Y+S)\to \Coh^1(S,(-rK_Y+S)|_S)\to \Coh^2(Y,-rK_Y)\to \Coh^2(Y,-rK_Y+S)\to\cdots$$
Since $S$ is ample and $-rK_Y$ is nef, by the Fujita vanishing theorem (or Kawamata-Viehweg vanishing theorem), we have $\Coh^i(Y,-rK_Y+S)=0$ for $i>0$.
Therefore,  $\Coh^1(S,(-rK_Y+S)|_S)\simeq\Coh^2(Y,-rK_Y)$ and our lemma is thus proved.
\end{proof}

\begin{lemme}\label{lem-gamma0-nu1}
If $\Gamma=0$ and $-K_Y$ is not big, then the numerical dimension  $\nu(Y,-K_Y)$ of $-K_Y$ is $1$.
\end{lemme}
\begin{proof}
If the numerical dimension $\nu(Y,-K_Y)=3$, then $-K_Y$ and hence $-(K_X+\Delta)$ are big, a contradiction to our assumption (cf.~{\hyperref[lem-big-ample]{Lemma \ref*{lem-big-ample}}}).
Clearly, $-K_Y\not\equiv0$, and thus we only need to exclude the case $\nu(Y,-K_Y)=2$.
Take a sufficiently ample general hypersurface $S$ on $Y$.
Then we have $K_Y^2\cdot S\neq 0$ since $K_Y^2\not\equiv 0$.
Therefore $-K_Y|_S$ is a nef and big divisor on $S$.
Applying Serre duality and Kawamata-Viehweg vanishing theorem, we have $\Coh^1(S,(-rK_Y+S)|_S)=0$ (where $r$ is the Cartier index of $K_Y$),  a contradiction to 
{\hyperref[lem-gamma0-h1=h2]{Lemma  \ref*{lem-gamma0-h1=h2}}} and 
{\hyperref[lem-gamma0-h2>0]{Lemma  \ref*{lem-gamma0-h2>0}}}. 
\end{proof}

\begin{thm}\label{thm-case-gamma=0}
If $\Gamma=0$, then $-(K_X+\Delta)$ is ample.	
\end{thm}

\begin{proof}
Note that $\Delta=0$ in this case. 
In view of {\hyperref[lem-gamma0-nu1]{Lemma  \ref*{lem-gamma0-nu1}}}, we only need to show the case $\nu(Y,-rK_Y)=1$ is impossible.
Recall that $\tau:Y\to X$ is the $\mathbb{Q}$-factorial terminalization. 
Assume that  $\nu(Y,-rK_Y)=1$. 
Then, we have $(-K_Y|_S)^2=K_Y^2\cdot S=0$.
Denote by $U$ the regular locus of $Y$.
By {\hyperref[lem-finite-reg-algpi1]{Proposition   \ref*{lem-finite-reg-algpi1}}}, $\pi_1^{\textup{alg}}(U)$ is finite.
Take a sufficiently general very ample divisor $S$ on $Y$, which is not contracted by $\tau$. 
Then $\pi_1^{\textup{alg}}(U)\cong\pi_1^{\textup{alg}}(S\cap U)$ is finite (cf.~\cite[Theorem 1.1.3]{HL85}), noting that $U$ has the same homotopy type of the space obtained from $S\cap U$ by attaching cells of dimension $\ge 3$, but the fundamental group of a CW complex only depends on its $2$-skeleton (and hence $\pi_1(U)\cong\pi_1(S\cap U)$).
Since our $S$ is general and $Y$ has at worst terminal (and hence isolated) singularities (cf.~\cite[Corollary 5.18]{KM98}), $S\cap U$ coincides with $S$.
Therefore, $\pi_1^{\textup{alg}}(S)$ is finite.
On the one hand, since $\pi_1^{\textup{alg}}(S)$ is finite, applying \cite{Miy88} for the nef divisor $-(r+1)K_Y|_S$ with $(-(r+1)K_Y|_S)^2=0$ (cf. also \cite[Theorem 3.6]{Ser95}), we see that, either $\Coh^1(S,(r+1)K_Y|_S)=0$, or there exists a positive integer $n$ such that $\Coh^0(S,-(r+1)nK_Y|_S)\neq 0$.
On the other hand, it follows from {\hyperref[lem-gamma0-h1=h2]{Lemma  \ref*{lem-gamma0-h1=h2}}}, {\hyperref[lem-gamma0-h2>0]{Lemma  \ref*{lem-gamma0-h2>0}}} and Serre duality that
$\Coh^1(S,(r+1)K_Y|_S)=\Coh^1(S,(-rK_Y+S)|_S)\neq 0$.
In conclusion,  $-K_Y|_S$ is $\mathbb{Q}$-linearly equivalent to a non-zero effective divisor.
Let $T:=\tau(S)$, which is still a surface due to the choice of $S$. 
Then it follows from the commutative diagram that $-K_X|_T$ is $\mathbb{Q}$-linearly equivalent to a non-zero effective Weil $\mathbb{Q}$-divisor (as a curve) on $T$. 
%noting that $-(K_X+\Delta)|_T$ is still strictly nef.
Hence, it follows from the projection formula and from the strict nefness of $-K_X$ that $0<-K_X\cdot (-K_X|_T)=-K_Y\cdot (-K_Y|_S)=K_Y^2\cdot S=0$, which gives rise to a contradiction.
\end{proof}

%\begin{proof}
%In view of {\hyperref[lem-gamma0-nu1]{Lemma  \ref*{lem-gamma0-nu1}}}, we only need to show the case $\nu(Y,-rK_Y)=1$ is impossible.
%Assume that  $\nu(Y,-rK_Y)=1$. 
%Denote by $U$ the regular locus of $Y$.
%By {\hyperref[lem-finite-reg-algpi1]{Lemma  \ref*{lem-finite-reg-algpi1}}},  $\pi_1^{\textup{alg}}(U)$ is finite.
%Let $S$ be a general (very ample) hyperplane section of $Y$.
%Then $\pi_1^{\textup{alg}}(U)\cong\pi_1^{\textup{alg}}(S\cap U)$ is finite (cf.~\cite[Theorem 1.1.3]{HL85}), 
%noting that $U$ has the same homotopy type of the space obtained from $S\cap U$ by attaching cells of dimension $\geqslant 3$, 
%but the fundamental group of a CW complex only depends on its $2$-skeleton (and hence $\pi_1(U)\cong\pi_1(S\cap U)$). 
%Since our $S$ is general and $Y$ has at worst terminal (and hence isolated) singularities (cf.~\cite[Corollary 5.18]{KM98}), $S\cap U$ coincides with $S$.
%Therefore, $\pi_1^{\textup{alg}}(S)$ is finite. 
%Since $\nu(S,-K_Y|_S)=1$, we have $(-K_Y|_S)^2=K_Y^2\cdot S=0$.
%Now that $\pi_1^{\textup{alg}}(S)$ is finite, we can apply \cite[Theorem 2.1]{Miy88} (cf.~\cite[Theorem 3.6]{Ser95}) and then $-K_Y|_S$ is $\mathbb{Q}$-linearly equivalent to an effective divisor, noting that $H^1(S,(r+1)K_Y|_S)=H^1(S,-(r+1)K_Y|_S+K_S)=H^1(S,(-(r+1)K_Y+S)|_S)\neq 0$ (cf.~{\hyperref[lem-gamma0-h1=h2]{Lemma  \ref*{lem-gamma0-h1=h2}}} and 
%{\hyperref[lem-gamma0-h2>0]{Lemma  \ref*{lem-gamma0-h2>0}}}). 
%As a result, the nef dimension $\textup{nd}(-K_Y|_S)=2$, the numerical dimension $\nu(-K_Y|_S)=1$, and the Kodaira dimension $\kappa(-K_Y|_S)\geqslant 0$. 
%Then, it follows from \cite[Proposition 6.1]{Amb04} that $S$ is uniruled. 
%However, since our $S$ is sufficiently ample, our $K_S$ is also ample, a contradiction to \cite{BDPP13}.
%\end{proof}

\begin{proof}[Proof of \texorpdfstring{{\hyperref[mainthm-ample-3klt]{Theorem \ref*{mainthm-ample-3klt}}}}{text}]
It follows directly from {\hyperref[thm-case-gammaneq0]{Theorem \ref*{thm-case-gammaneq0}}} and  {\hyperref[thm-case-gamma=0]{Theorem \ref*{thm-case-gamma=0}}} (cf.~{\hyperref[rmq-case-gamma=0]{Remark \ref*{rmq-case-gamma=0}}}).
\end{proof}

\vskip 2\baselineskip

\begingroup
\setstretch{1.0}

\bibliographystyle{alpha}
\bibliography{strict-nef}

\endgroup
\end{document}